\RequirePackage[l2tabu, orthodox]{nag}

\documentclass[reqno]{amsart}
\usepackage{amssymb}
\usepackage{amsmath}
\usepackage{amsthm}
\usepackage{mathtools}
\usepackage[svgnames]{xcolor}
\usepackage[pdftex]{hyperref}
\usepackage{enumitem}
\usepackage[utf8]{inputenc}

\usepackage{vmargin}
\setmarginsrb {2cm}{2cm}{2cm}{2cm} {1cm}{1cm}{0.5cm}{0.5cm}

\hypersetup{
  pdftitle={},
  pdfauthor={Stefan Le Coz <slecoz@math.univ-toulouse.fr>},
  pdfsubject={},
  pdfkeywords={}, 
  colorlinks=true,
  linkcolor=DarkBlue  ,          
  citecolor=DarkRed,        
  filecolor=DarkMagenta,      
  urlcolor=DarkGreen,           
}

\newtheorem{theorem}{Theorem}[section]
\newtheorem{proposition}[theorem]{Proposition}
\newtheorem{corollary}[theorem]{Corollary}
\newtheorem{lemma}[theorem]{Lemma}

\theoremstyle{definition}
\newtheorem{definition}[theorem]{Definition}

\theoremstyle{remark}
\newtheorem{remark}[theorem]{Remark}

\DeclarePairedDelimiter{\norm}{\lVert}{\rVert}

\newcommand{\scalar}[2]{\left( #1,#2 \right)}

\newcommand{\dual}[2]{\left\langle #1,#2 \right\rangle}

\newcommand{\eps}{\varepsilon}

\newcommand{\R}{\mathbb{R}}

\DeclareMathOperator{\sech}{sech}
\DeclareMathOperator{\inertia}{Inertia}
\DeclareMathOperator{\Span}{span}
\newcommand*\rmd{\mathop{}\!\mathrm{d}}

\renewcommand{\leq}{\leqslant}
\renewcommand{\geq}{\geqslant}

\DeclareMathAlphabet{\mathpzc}{OT1}{pzc}{m}{it}

\usepackage{color}

\begin{document}

\title[Stability of the multi-solitons of mKdV]{Stability of the
  multi-solitons \\of the modified Korteweg-de Vries equation}

\author[S.~Le Coz]{Stefan Le Coz}
\thanks{The work of S. L. C. is 
  partially supported by ANR-11-LABX-0040-CIMI within the
  program ANR-11-IDEX-0002-02 and  ANR-14-CE25-0009-01}

\author[Z.~Wang]{Zhong Wang}
\thanks{The work of Z. W. is
 supported by the China National Natural Science Foundation under grant number
11901092, Guangdong Natural Science Foundation under grant number 2017A030310634 and a scholarship  from China Scholarship Council (no. 201708440461). }

\address[Stefan Le Coz]{Institut de Math\'ematiques de Toulouse ; UMR5219,
 \newline\indent
  Universit\'e de Toulouse ; CNRS,
 \newline\indent
  UPS IMT, F-31062 Toulouse Cedex 9,
 \newline\indent
  France}
\email[Stefan Le Coz]{stefan.lecoz@math.cnrs.fr}

\address[Zhong Wang]{School of Mathematics and Big Data,
 \newline\indent
  Foshan University,
 \newline\indent
  Foshan, Guangdong, {528000},
 \newline\indent
P. R. China.}
\email[Zhong Wang]{ wangzh79@mail2.sysu.edu.cn}


\date{\today}

\begin{abstract}
We establish the nonlinear stability of $N$-soliton solutions of the
modified Korteweg-de Vries (mKdV) equation. The $N$-soliton solutions
are global solutions of mKdV  behaving at (positive and negative) time
infinity as sums of $1$-solitons with speeds $0<c_1<\cdots< c_N$.
The proof relies on the variational characterization of
$N$-solitons. We show that the $N$-solitons realize the local
minimum of the $(N+1)$-th mKdV conserved quantity subject to fixed
constraints on the $N$ first conserved quantities.
To this aim, we  construct a functional for which $N$-solitons are critical points,
we prove that the spectral properties of the linearization of this functional around a $N$-soliton are preserved on the extended timeline, and we analyze the spectrum at infinity of linearized operators around $1$-solitons. 
The main new ingredients in our analysis are a new operator identity based on a generalized Sylvester law of inertia and recursion operators for the mKdV equation.
\end{abstract}

\maketitle

 \setcounter{tocdepth}{1}

\tableofcontents

\section{Introduction}
\label{sec:introduction}

We consider the modified Korteweg-de Vries equation
\begin{equation}
 \label{eq:mkdv}
 \tag{mKdV}
u_t+(u_{xx}+u^3)_x=0,
\end{equation}
where $u:\R_t\times\R_x\to\R$. The modified Korteweg-de Vries equation~\eqref{eq:mkdv} is a well-known completely integrable model~\cite{Mi68,Wa73}. In particular, solutions might be constructed using
the inverse scattering transform and there exists an infinite sequence
of conservations laws.

Among the possible solutions of~\eqref{eq:mkdv}, some are of particular interest: the solitons and
multi-solitons. A soliton is a solution of the form
\[
U_{c_1}(t,x)=Q_{c_1}(x-c_1t-x_1),
\]
where the profile $Q_{c_1}$ is fixed along the time evolution and is
translated along $\R$ at speed $c_1>0$ with initial position $x_1$. A
multi-soliton is a solution $U_{c_1,\dots,c_N}$ of~\eqref{eq:mkdv} 
such that
\[
U_{c_1,\dots,c_N}(t,x)\sim_{t\to\pm\infty}\sum_{j=1}^NQ_{c_j}(x-c_jt-x_j^\pm),
 \]
which means that $U_{c_1,\dots,c_N}$ behaves at negative and positive time
infinity as a sum of solitons. Explicit formulas for solitons and
multi-solitons are known and will be recalled in Section~\ref{sec:preliminaries}.

It has long been known (see Schuur~\cite{Sc86}) that a solution of the
classical Korteweg-de Vries equation (i.e. when the nonlinearity is quadratic
instead of cubic) decomposes as a finite sum of solitons and a
dispersive remainder. This type of behavior is expected to be generic
for nonlinear dispersive equations, but it has seldom been
rigorously established and remains known most of the time under the
name \emph{Soliton Resolution Conjecture}. In the case of the modified
Korteweg-de Vries equation, the conjecture has been established
recently in weighted spaces and for multi-solitons in~\cite{ChLi19}. However, whereas for the
classical Korteweg-de Vries equation the only nonlinear  solutions
obtained via inverse scattering are the multi-solitons, for the
modified Korteweg-de Vries equation the inverse scattering also
generates breathers and $N$-poles (see~\cite{Wa73,WaOh82}), which are not yet taken into
account by any soliton resolution statement. Observe  that~\eqref{eq:mkdv} possesses even more complicated solutions like
self-similar solutions (see~\cite{CoCoVe20} for their asymptotic
behavior in Fourier space).

One of the major questions related to multi-solitons is their stability
with respect to the dynamics of the equation. In the case of the
classical Korteweg-de Vries equation, this question was settled in 1993
by Maddocks and Sachs~\cite{MaSa93}: $N$-solitons are stable in
$H^N(\R)$. Our goal in this paper is to establish the counter-part of this result in the case of the modified Korteweg-de Vries equation.

Our main result is the following.

\begin{theorem}
 \label{thm:stability}
  Given $N\in\mathbb N,$ $N\geq1$, a collection of speeds $\mathbf c=(c_1,\dots,c_N)$ with $0<c_1<\cdots<c_N$ and a collection of phases $\mathbf x=(x_1,\dots,x_N)\in\R^N$, let $U_{\mathbf c}^{(N)}(\cdot,\cdot;\mathbf x)$ be the corresponding multi-soliton given by~\eqref{eq:def-multi-solitons}. For any $\eps>0$, there exists $\delta>0$ such that  for any $u_0\in H^N(\R)$, the following stability property holds. If
 \[
\norm{u_0-U_{\mathbf c}^{(N)}(0,\cdot,\mathbf x)}_{H^N}<\delta,
\]
then for any $t\in\R$ the corresponding solution $u$ of~\eqref{eq:mkdv} verifies
 \[
\inf_{\tau\in\R,\mathbf y\in\R^N}\norm{u(t)-U_{\mathbf c}^{(N)}(\tau,\cdot,\mathbf y)}_{H^N}<\eps.
\]
\end{theorem}

Some discussion of the notion of stability obtained in Theorem~\ref{thm:stability} is in order, as many possible notions of stability exist, already for single solitons, and even more in the case of multi-solitons.  Observe that for the comprehension we have neglected in the statement  of Theorem~\ref{thm:stability} a redundancy in the stability expression, as we in fact have
 \[
 \{U_{\mathbf c}^{(N)}(\tau,\cdot,\mathbf y):\tau\in\R,\mathbf y\in\R^N\}=
 \{U_{\mathbf c}^{(N)}(0,\cdot,\mathbf y):\mathbf y\in\R^N\}.
 \]
Our stability statement is valid for the whole timeline, from infinity in the past to infinity in the future. This feature is usually specific to integrable equations, we should comment later on stability statements obtained for only one end of the timeline in non-integrable models. The stability statement could be reformulated in terms of stability of a set in the following way. A set is said to be \emph{stable} if any solution with initial data close to this set will remain close to this set for all time. Different kind of sets can be considered, for example the time orbit of the multi-soliton, the family of multi-soliton profiles (with same speeds), the set of (local or global) minimizers of some variational problems. For solitons of~\eqref{eq:mkdv}, it is known that these three sets coincide. However, it is not always the case. In particular, the first two sets are different as soon as we consider $N$-solitons with $N\geq 2$, and our stability result concerns the second set. It is indeed not hard to verify using the explicit formula of the $N$-solitons that the time orbit of the $N$-solitons cannot be stable (to make our result a time-orbit stability result, one would need to include all possible time-orbits under the  $N$ first Hamiltonian flows of the~\eqref{eq:mkdv} hierarchy, see e.g. the discussion in~\cite[p. 869]{MaSa93}). A typical result of stability of the third kind of sets (i.e. sets of minimizers) is the seminal work of  Cazenave and Lions~\cite{CaLi82}. The flexibility and versatility of variational technics makes the
stability of this kind of sets easier to obtain, but leads to potentially weaker stability statements unless some uniqueness or non-degeneracy of the minimizers is established. Unfortunately, uniqueness statements are most of the time widely open problems (for more in this direction, see the recent work of Albert~\cite{Al19} in the case of the classical Korteweg-de Vries equation for a uniqueness result for the two-solitons). In our case, we are able to obtain the non-degeneracy property in the same process as a local minimization property. 

Observe here that, while solutions behaving at both ends of the time line as pure sums of solitons are probably bound to exist only in integrable cases, it is nevertheless possible to obtain multi-soliton solutions for  non-integrable equations if the behavior is expected only at positive (or negative) large times.
In the framework  of the
nonlinear Schr\"odinger equation, in 1990, Merle~\cite{Me90} obtained a first
existence result  for the mass-critical case. Since then, many existence results for multi-solitons have been obtained
in different settings (see~\cite{BeGhLe14,CoLe11,CoMaMe11,LeLiTs15,LeTs14,MaMe06,WaCu15,Wa16,WaCu17,Wa17a} among many others).  In the
framework of Korteweg-de Vries type equations, existence (and
uniqueness) of multi-solitons in non-integrable cases was first
established by Martel~\cite{Ma05}. Stability of multi-solitons for generalized Korteweg-de Vries
equations was obtained by Martel, Merle and Tsai in~\cite{MaMeTs02} (see
also~\cite{AlBoNg07} for related developments). Using a similar
approach, some stability results have been obtained in the nonlinear
Schr\"odinger case (see~\cite{MaMeTs06} and more recently~\cite{LeWu18}), but the results are only partial and stability of
multi-solitons remains essentially an open problem in the Schr\"odinger case. In
the case of the classical Korteweg-de Vries equation, results
combining the approaches of~\cite{MaSa93} and~\cite{MaMeTs06} have
been obtained by Alejo, Mu\~{n}oz and Vega~\cite{AlMuVe13}, with in particular results of
stability and asymptotic stability in $L^2(\R)$ for multi-solitons. A
detailed overview of these results is offered by Mu\~{n}oz in~\cite{Mu14}.
Let us also mention the asymptotic stability results obtained for generalized
Korteweg-de Vries equations in~\cite{GePuRo16,PeWe94}.

The premises of the stability analysis of $N$-solitons  may be found in the pioneering work of Lax~\cite{La68}, in which in particular the variational principle satisfied by multi-solitons of the Korteweg-de Vries equation is given. However, it is Maddocks and Sachs~\cite{MaSa93} who laid the cornerstone for the stability
analysis of multi-solitons in integrable equations. Their approach relies
essentially on spectral and variational arguments, and makes no (direct) use of inverse
scattering. The integrable nature of the equation is used essentially
in two aspects: first, for the explicit formulas for multi-solitons,
second for the construction of an infinite sequence of conservation
laws. Indeed, the central point of~\cite{MaSa93} is to characterize
$N$-solitons as (local) minimizers of the $(N+1)$-th conserved quantity subject
to fixed constraints on the $N$ first conserved
quantities. In a way, this argument is to be related to the theories
developed by Benjamin, Bona, Grillakis, Shatah and Strauss~\cite{Be72,Bo75,GrShSt87,GrShSt90} for the stability of
a single solitary wave. 

The ideas developed by Maddocks and Sachs have been successfully implemented to obtain stability results in various settings.
Neves and Lopes~\cite{NeLo06} proved the stability of the two-solitons of the Benjamin-Ono equation. Alejo and Mu\~{n}oz~\cite{AlMu13} established the stability of~\eqref{eq:mkdv} breathers (which can be formally seen as counterparts of two-solitons for complex speeds). 
Spectral stability for multi-solitons in the KdV hierarchy was considered by Kodoma and Pelinovsky~\cite{KoPe05}.
We also mention the work of Kapitula~\cite{Ka07}, which is devoted to the stability of $N$-solitons of a large class of integrable systems, including in particular the model cubic nonlinear Schr\"odinger equation. Very recently, a variational approach was used by Killip and Visan~\cite{KiVi20} to obtain the stability of multi-solitons of the classical Korteweg-de Vries equation in weak regularity spaces  (up to $H^{-1}(\R)$ !).
Finally, a stability result in low regularity  $H^s$-spaces was also obtained very recently by Koch and Tataru~\cite{KoTa20} for the multi-solitons of both modified Korteweg-de Vries equation and the cubic nonlinear Schr\"odinger equation. This result contains ours, as it is valid in particular for $s=N$. The proof is however much more involved and relies on a extensive analysis of an iterated B\"acklund transform.

The major difference between our approach and the approach of Maddocks and Sachs lies in the analysis of spectral properties. In particular, we develop in the context of~\eqref{eq:mkdv}, and for $N$-solitons, ideas introduced by Neves and Lopes~\cite{NeLo06} for the analysis of the two-solitons of the Benjamin-Ono equation. Indeed, the spectral analysis of Maddocks and Sachs and many of their continuators relies on an extension of Sturm-Liouville theory to higher order differential equations (see~\cite[Section 2.2]{MaSa93} and~\cite{Gr91}). As the Benjamin-Ono equation is non-local, Neves and Lopes~\cite{NeLo06} were lead to introduce a new strategy relying on iso-inertial properties of linearized operators. It turns out that this type of argument can also be implemented for local problems such as~\eqref{eq:mkdv}. Our first task was to extend the spectral theory of Neves and Lopes~\cite{NeLo06} to an arbitrary number $N$ of composing solitons. Apart from an increased technical complexity (inherent to the fact that the number of composing solitons is now arbitrary), no major difficulty arises here. Then our second task was to implement this spectral theory for the multi-solitons of~\eqref{eq:mkdv}. At that level, we had to overcome major obstacles. Most of the existing works content themselves with the simpler analysis of two-solitons, for which many informations can be obtained by brut force (it is said in~\cite{NeLo06}: ``It is likely that our method can be extended to multi-solitons of the BO equation and of its hierarchy but the algebra may become prohibitive''). Hence, to deal with the arbitrary $N$ case, it was necessary to acquire a deeper understanding of the relationships between $N$-solitons, the variational principle that they satisfy, and the spectral properties of the operators obtained by linearization of the conserved quantities around them. 

We now present the process leading to the proof of our main result Theorem~\ref{thm:stability}.  

We first review in Section~\ref{sec:preliminaries} the results gravitating around our main topic of interest. We recall the well-posedness of the Cauchy problem, and remind the reader that the conservation laws for~\eqref{eq:mkdv} may be obtained from one another using a \emph{recursion formula} (see~\eqref{eq:recursion}) involving the first derivative of consecutive conservation laws and what we call the \emph{recursion operator} $\mathcal K$ (see~\eqref{eq:skew K}). We also recall the formulas for solitons and multi-solitons.

Section~\ref{sec:vari-princ} is devoted to the next step: establishing the variational principle verified by the multi-solitons, i.e. to construct a functional $S_N$ of which $N$-solitons are critical points. The form of the variational principle as well as some elements of proof were given by Lax~\cite{La68}. Holmer, Perelman and Zworski~\cite{HoPeZw11} later established a rigorous proof for the $2$-solitons, which we adapt here to the case of $N$-solitons. The proof proceeds into two steps. First, as $N$-solitons are decomposing at time infinity as decoupled solitons, the variational principle that they possibly satisfy should also be verified by each of their composing solitons. As a consequence, the coefficients of the variational principle are determined by the speeds of the composing solitons. Second, we prove that the $N$-solitons indeed verify the conjectured variational principle by a rigidity argument on the differential equation verified by a remainder term. The proof given here is analytic in spirit and makes little use of the algebraic structure of the problem. Alternative strategies to obtain a similar result using the inverse scattering approach are possible, see e.g.~\cite{Ka07,LiWa20}.

Given the functional $S_N$ admitting a $N$-soliton as critical point, we hold a natural candidate for a Lyapunov functional allowing to prove stability. Indeed, it was proved by  Maddocks and Sachs that if one can equate the number of negative eigenvalues of the operator corresponding to the Hessian with the number of positive principal curvatures of the solution surface (see Proposition~\ref{prop:7.1.} or~\cite[Lemma 2.3]{MaSa93}), then a Lyapunov functional based on an augmented Lagrangian may be constructed and stability follows (the reader familiar with the stability theory of single solitons will recognize in these two criteria the equivalent for multi-solitons of the spectral and slope conditions rendered famous  by Grillakis, Shatah and Strauss~\cite{GrShSt87}). The spectral analysis represents the major task and is spread on two sections.

At first, in Section~\ref{sec:spectral-theory}, one needs to extend to the $N$-soliton case the theory developed by Neves and Lopes~\cite{NeLo06} in the case of $2$-solitons. Indeed, in the spectral analysis of linearized operators, a major difference appears between solitons and multi-solitons: whereas it is possible for solitons to consider the perturbation at the profile level and therefore to work with operators having time independent potentials, the operators associated with multi-solitons have inherently time dependent potentials. To overcome this difficulty, and somehow to go back to time-independent potentials, one needs a relation between the spectral structure along the time evolution and the spectral structure at time infinity (where the decoupling between solitons brings us back to the case of $1$-solitons). This comes in the form of the \emph{preservation of inertia property}, i.e. the numbers of negative and zero eigenvalues are constant along the extended timeline (see Proposition~\ref{prop:iso-inertia} and Corollary~\ref{cor:asymptotic-inertia}).

With this tool in hand, the spectral analysis is obtained as the spectral analysis of the linearized operator at infinity, which is itself the combination of the spectral
analysis of the linearized operators around each of the composing solitons. In Section~\ref{sepecsec}, the later analysis is made possible by a remarkable factorization identity (see Proposition~\ref{prop:factorization_L_N_j}), which we obtain thanks to the recursion properties of the linearized conserved quantities around each soliton. Indeed, given $Q_j$ the $j$-th soliton profile, one may introduce the operators 
\[
  M_j=Q_j\partial_x\left(\frac{\cdot}{Q_j}\right),\quad M_j^t=\frac{1}{Q_j}\partial_x\left(Q_j\,\cdot\,\right),
\]
and, denoting the linearized operator around $Q_j$ by $L_{N,j}:=S_N''(Q_j)$, we have
\[
M_jL_{N,j}M_j^t=M_j^t\left(\prod_{k=1}^N(-\partial_x^2+c_k)\right)M_j,
\]
which allows us to obtain the necessary spectral informations.

Finally, in Section~\ref{stasec7}, we compute the number of positive principal curvatures for the multi-soliton surface by an astute use of the (matrix) Sylvester's law of inertia combined with the relations between the coefficients of the candidate Lyapunov functional and the speeds of the multi-soliton. The stability of the $N$-soliton is then a consequence of the combination of the previous arguments.

\section{Preliminaries}
\label{sec:preliminaries}

In this section we collect some preliminary results on~\eqref{eq:mkdv}.

\subsection{Hamiltonian structure and conserved quantities}

The first few conserved quantities of~\eqref{eq:mkdv} are given by
\begin{align}
\text{(mass)} \quad H_0(u)&:=\int_{\mathbb R}u\rmd x,\nonumber\\
\text{(momentum)} \quad H_1(u)&:=\frac12\int_{\mathbb R}u^2\rmd x,\label{momentum}\\
\text{(energy)} \quad H_2(u)&:=\frac{1}{2}\int_{\mathbb R} u_x^2-\frac{1}{4}\int_{\mathbb R}u^{4}\rmd x, \label{ener}\\
 \text{(second energy)} \quad  H_3(u)&:=\frac{1}{2}\int_{\mathbb R}u_{xx}^2\rmd x+\frac{1}{4}\int_{\mathbb R}u^{6}\rmd x-\frac52\int_{\mathbb R}u^2u_x^2\rmd x. \nonumber
\end{align}
In general, for $n\in\mathbb N$, the conserved quantities of~\eqref{eq:mkdv} are of the form
\begin{equation*}
H_n(u):=\frac12\int_{\mathbb R}u_{(n-1)x}^2\rmd x+\int_{\mathbb R}q_n(u, u_x,\dots,u_{(n-2)x} )\rmd x,
\end{equation*}
where $q_n$ is a polynomial which might be explicitly calculated. Various strategies are possible to generate the conserved quantities of~\eqref{eq:mkdv}. In particular, one might rely on the following Lenard recursion identity. For $u\in\mathcal S(\R)$ (the Schwartz space of fast-decaying smooth functions), define the \emph{recursion operator} $\mathcal K$ by
\begin{equation}
 \label{eq:skew K}
\mathcal K(u):=-\partial_x^3-2u^2\partial_x-2u_x\partial_x^{-1}(u\partial_x),\quad \partial_x^{-1}u:=\frac12\left(\int_{-\infty}^xu(y)\rmd y-\int_x^{\infty}u(y)\rmd y\right).
\end{equation}
For all $n\geq0$, we have the recursion formula (see~\cite{Ol86} or~\cite[formula (2.4)]{HoPeZw11})
\begin{equation}
 \label{eq:recursion}
\partial_x H_{n+1}'(u)= \mathcal K(u) H_n'(u).
\end{equation}
The modified Korteweg-de Vries equation~\eqref{eq:mkdv} is a Hamiltonian system of the form
\[
u_t=\partial_x H_2'(u).
\]
The recursion formula readily leads to another Hamiltonian expression for~\eqref{eq:mkdv}:
\[
u_t=\mathcal K(u) H_1'(u).
\]
This bi-Hamiltonian nature allows to consider the \emph{mKdV hierarchy}, a  generalized
class of equations given by
\[
u_t=\partial_x H_{n+1}'(u)=\mathcal{K}(u) H_{n}'(u), \quad n\in\mathbb{N}.
\]
In particular, the functionals $H_n$ are constant along the flow of all equations in the hierarchy.

A substantial body  of works is available regarding the Cauchy problem for the modified Korteweg-de Vries equation~\eqref{eq:mkdv}. In particular, one may refer to the celebrated works of Kenig, Ponce and Vega~\cite{KePoVe93} and Colliander, Keel, Staffilani, Takaoka,  and Tao~\cite{CoKeStTaTa03}, or see some of the recent books on the topics~\cite{KoTaVi14,LiPo15,Zh01}.
In this work, we will make use of the following property, which has been established in a streamlined proof (using only the necessary elements of~\cite{KePoVe93}) by Holmer, Perelman  and Zworski~\cite{HoPeZw11}. For all $k\in\mathbb{N}$, given any initial data $u_0\in H^k(\R)$ there exists a unique global solution $u\in\mathcal C(\R,H^k(\R))$ of~\eqref{eq:mkdv} such that $u(0)=u_0$. Moreover, the data-to-solution map is continuous and $H_j(u)$ is preserved by the flow  for $j=1,\dots,k+1$. 

\subsection{Solitons and Multi-solitons}

The inverse scattering method allows, by purely algebraic technics, to calculate explicitly
solutions of~\eqref{eq:mkdv} (at least for rapidly decreasing
solutions) and we now give a quick review of some solutions which have
been constructed for~\eqref{eq:mkdv}. Details of the constructions are
given in~\cite{Hi72,Wa73,WaOh82}. Recent progress using the inverse scattering approach (including a soliton resolution result and asymptotic stability of multi-solitons in weighted spaces) are reported in~\cite{ChLi19}.

We start with the solitons. A \emph{soliton} of~\eqref{eq:mkdv} is a
traveling wave solution of the form
\[
u(t,x)=Q_c(x-ct+x_0),
\]
where $c\in\R$ is the speed and $x_0$ is the initial position. The
profile $Q_c$ satisfies the ordinary differential equation
 \begin{equation}
 \label{eq:stationary}
-\partial_{xx}Q_c+cQ_c-Q_c^3=0.
\end{equation}
The soliton profile $Q_c$ can be proved to be a minimizer of the energy $H_2$ (see~\eqref{ener}) under the momentum (see~\eqref{momentum}) constraint $H_1(u)=H_1(Q_c)=2\sqrt{c}$. Up to sign change and translation, there exists a
unique positive even solution to the profile equation~\eqref{eq:stationary}, which is
explicitly given by the formula
 \begin{equation}
 \label{eq:Q}
Q_c(x)=\sqrt{c}Q(\sqrt{c}x),\quad Q(x)=\sqrt{2}\sech(x).
\end{equation}
To make a link with what follows, note that the $1$-soliton with speed
$c_1$ and shift parameter $x_1$ can be
written in the form
\[
U_{c_1}(t,x;x_1)=2\sqrt{2}\partial_x\left(\arctan\left( e^{s_1} \right)\right),
\]
where $s_1=\sqrt{c_1}(x-c_1t)+x_1$.

Solitons form the building blocks for more complicated dynamics of~\eqref{eq:mkdv}, which we now present, starting with $2$-solitons.

Given speeds $c_1,c_2>0$, $c_1\neq c_2$ and shift parameters
$x_1,x_2\in\R$, a \emph{$2$-soliton} is a solution of~\eqref{eq:mkdv} given
by 
 \begin{equation}
U_{c_1,c_2}(t,x;x_1,x_2)=2\sqrt{2} \partial_x \left( \arctan \left( \frac{e^{s_1}+e^{s_2}}{1-\rho^2e^{s_1+s_2}}\right) \right),\label{eq:2soliton}
\end{equation}
where $s_j:=\sqrt{c_j}(x-c_jt)+x_j$ for $j=1,2$, and $\rho:=\frac{\sqrt{c_1}-\sqrt{c_2}}{\sqrt{c_1}+\sqrt{c_2}}$. 
Asymptotically in time, this solution decomposes into a sum of two
$1$-solitons traveling at speeds $c_1$ and $c_2$. More precisely,
there exist $x_1^\pm,x_2^\pm$ depending explicitly on $c_1$, $c_2$, $x_1$, $x_2$
such that 
\begin{equation*}
 \lim_{t\to\pm\infty}\norm{U_{c_1,c_2}(t,\cdot;x_1,x_2)-Q_{c_1}(\cdot-c_1t-x_1^\pm)-Q_{c_2}(\cdot-c_2t-x_2^\pm)}_{H^1}=0.
\end{equation*}
As can be observed in the above formula, in the $2$-solitons the interaction between the
two composing solitons is smooth and its only consequence is a shift
in the trajectories, as $x_j^-\neq x_j^+$ for $j=1,2$.

Observe here that  when $c_1=c_2$, there exist also solutions behaving at time infinity as two solitons traveling at the same speed and going away at logarithmic rate (see~\cite{WaOh82}). Those solutions, called double-poles, are however given by a formula different from~\eqref{eq:2soliton} and are not included in the results of the present paper. Our progress in the analysis of such solutions will be reported in a future work.

The formula for $N$-solitons for generic $N$ is slightly more complicated but has a similar form.

Given $N\in\mathbb N$, speeds $0<c_1<\cdots<c_N$, phases $x_1,\dots,x_N\in\R$, a $N$-soliton solution is given by
 \begin{equation}
 \label{eq:def-multi-solitons}
U_{c_1,\dots,c_N}(t,x;x_1,\dots,x_N)=2\sqrt{2}\partial_x \left(
 \arctan \left( \frac{g(t,x)}{f(t,x)}\right)\right),
\end{equation}
where  the functions $f$ and $g$ are given by 
 \begin{align*}
f(t,x)&=\sum_{n=0}^{\left\lfloor\frac N2\right\rfloor}\sum_{\sigma\in\mathfrak{C}_{2n}^N}a(\sigma)\exp\left( s_{\sigma(1)}+\cdots+s_{\sigma(2n)} \right),\\
  g(t,x)&=\sum_{n=0}^{\left\lfloor\frac{ N-1}2\right\rfloor}\sum_{\sigma\in\mathfrak{C}_{2n+1}^N}a(\sigma)\exp\left( s_{\sigma(1)}+\cdots+s_{\sigma(2n+1)} \right).
 \end{align*}
Here, $\left\lfloor\frac N2\right\rfloor$ denotes the integer part of $\frac N2$ and $\mathfrak{C}_{2n}^N$ is the set of all possible combinations of $2n$ elements among $N$. The variables $s_j$ are given for $j=1,\dots,N$ by
\[
s_j:=\sqrt{c_j}\left(x-c_jt\right)+x_j.
\]
The function $a$ is build upon 
the functions $\tilde a$  given by
\[
\tilde a(k,l):=-\left(\frac{\sqrt{c_l}-\sqrt{c_k}}{\sqrt{c_l}+\sqrt{c_k}}\right)^2,
\]
and for $n\geq1$ and $\sigma:=\{i_1,\dots,i_{2n}\}$, we set
 \[
a(\sigma):=\prod_{1\leq k<l\leq 2n}\tilde a(i_k,i_l)
\]
and $a(\sigma)=1$ otherwise (i.e. if $\sigma$ is not in the above form).

It was shown in~\cite{Hi72} that the $N$-soliton solutions given by the above formula decompose at positive and negative time infinity as sums of solitons. As was shown by Martel~\cite{Ma05}, they are the unique solutions of~\eqref{eq:mkdv} having this prescribed behavior.

\section{The Variational Principle}
\label{sec:vari-princ}

We analyze in this section the variational principle satisfied by
multi-solitons.  

We first observe that the differential equation~\eqref{eq:stationary} verified by the soliton profile and the recursion relation~\eqref{eq:recursion} imply that the $1$-soliton $U_{c_1}(t)\equiv U_{c_1}(t,\cdot;x_1)$ with speed $c_1>0$ satisfies for all $n\geq 1$ \emph{and for any $t\in\R$} the following variational principle
\begin{equation}\label{eq:1-sol variaprinciple}
H_{n+1}'(U_{c_1}(t))+c_1 H'_{n}(U_{c_1}(t))=0.
\end{equation}
For future reference,
we calculate here the  quantities $H_j(Q_{c_1})$ related to the $1$-soliton profile $Q_{c_1}$. Multiplying~\eqref{eq:1-sol variaprinciple} with $\frac{\rmd Q_{c_1}}{\rmd {c_1}}$, for each $j$, we get
\[
\frac{\rmd H_{j+1}(Q_c)}{\rmd c}_{|c=c_1}=-c_1\frac{\rmd H_{j}(Q_c)}{\rmd c}_{|c=c_1}=\cdots=(-c_1)^j\frac{\rmd H_{1}(Q_c)}{\rmd c}_{|c=c_1}=(-1)^jc_1^{\frac{2j-1}{2}},
\]
and therefore
\begin{equation}\label{eq:jHamliton}
H_{j+1}(Q_{c_1})=(-1)^j\frac{2}{2j+1}c_1^{\frac{2j+1}{2}}.
\end{equation}

It can be verified by explicit calculations that the $2$-soliton $U_{c_1,c_2}(t) \equiv U_{c_1,c_2}(t,\cdot,x_1,x_2)$ with speeds $0<c_1< c_2$ satisfies for all $n\geq 1$ \emph{and for any $t\in\R$} the variational principle
\begin{equation*}
H_{n+2}'(U_{c_1,c_2}(t))+(c_1+c_2) H_{n+1}'(U_{c_1,c_2}(t))+c_1c_2H_{n}'(U_{c_1,c_2}(t))=0.
\end{equation*}

Using the explicit expression~\eqref{eq:def-multi-solitons} for the $N$-solitons, it would in theory be possible to verify by hand for any given $N$ that they also satisfy variational principles. Calculations would however rapidly become unmanageable when $N$ grows. In the following, we provide an analytic proof that the multi-solitons indeed verify a variational principle.
This fact is
commonly accepted but rarely proved. We base here our proof on the
approach outlined by Lax~\cite{La68} and later rigorously developed by Holmer, Perelman and
Zworski~\cite{HoPeZw11}.

\begin{proposition}
 \label{prop:vari-princ}
  Let $U:\R_t\times\R_x\to\R$ be a solution of~\eqref{eq:mkdv} and assume
  that there exist $N\in\mathbb N\setminus\{0\}$, $0<c_1<
 \cdots< c_N$, and $x_1,\dots,x_N:\R_t\to\R$ such that
 \[
\norm[\bigg]{U(t)-\sum_{j=1}^NQ_{c_j}(\cdot-x_j(t))}_{H^{N+1}}\lesssim e^{-\frac12\sqrt{c_1}|\min_{j,k}(x_j(t)-x_k(t))|},
\]
and for all $j=1,\dots,N$, we have
\begin{equation}
 \label{eq:9}
|\partial_tx_j(t)-c_j|\lesssim \frac1t.
\end{equation}
Then there exist $\lambda_1,\dots,\lambda_{N}$ such that
for all $t\in\R$ the function $U(t)$ verifies the variational principle
 \begin{equation}
H_{N+1}'(U(t))+\sum_{j=1}^N\lambda_{j}H_j'(U(t))=0.\label{eq:n-soliton Euler}
\end{equation}
The coefficients $\lambda_j$, $j=1,\dots,N$ are uniquely
determined in terms of the speeds $c_j$, $j=1,\dots,N$. Precisely,
they are given by Vieta's formulas: for $k=1,\dots,N$ we have
 \begin{equation}
\lambda_{N+1-k}=\sum_{1\leq i_{1}<\cdots<i_{k}\leq
  N}\left(\prod_{j=1}^{k}c_{i_j}\right).\label{eq:vieta}
\end{equation}
\end{proposition}

Let $\lambda_1,\dots,\lambda_N$ be given by~\eqref{eq:vieta}. 
For $u\in H^{N}(\R)$, we define the functional whose first derivative gives~\eqref{eq:n-soliton Euler} by
 \begin{equation}
 \label{eq:def-S_N}
S_N(u)=H_{N+1}(u)+\sum_{j=1}^N\lambda_{j}H_j(u).
\end{equation}

We first prove that if a solution of~\eqref{eq:mkdv} decomposes
asymptotically as a sum of solitons, then the parameters of the
variational principle it possibly satisfies are constrained
by the values of the speeds in the asymptotic decomposition and must satisfies~\eqref{eq:vieta}.

\begin{lemma}\label{lem:vieta}
  Let $U:\R_t\times\R_x\to\R$ be a solution of~\eqref{eq:mkdv} and assume
  that there exist $N\in\mathbb N\setminus\{0\}$, $0<c_1\leq
 \cdots\leq c_N$, and $x_1,\dots,x_N:\R_t\to\R$ such that
 \[
\lim_{t\to\pm\infty}\norm[\bigg]{U(t)-\sum_{j=1}^NQ_{c_j}(\cdot-x_j(t))}_{H^{N}}=0,
\]
and for all $j,k=1,\dots,N$, $j\neq k$ we have
\[
\lim_{t\to\pm\infty}|x_j(t)-x_k(t)|=\infty.
 \]
Assume also that there exist $\lambda_1,\dots,\lambda_{N}\in\R$ such that
for all $t\in\R$ the function $U(t)$ verifies the variational principle
 \begin{equation}
H_{N+1}'(U(t))+\sum_{j=1}^N\lambda_{j}H_j'(U(t))=0.\label{eq:vari-princ}
\end{equation}

Then the coefficients $\lambda_j$, $j=1,\dots,N$ are uniquely
determined in terms of the speeds $c_j$, $j=1,\dots,N$ by Vieta's formula~\eqref{eq:vieta}.
\end{lemma}

\begin{remark}
  The assumptions of Lemma~\ref{lem:vieta} are weaker than those of Proposition~\ref{prop:vari-princ}. In particular,  Lemma~\ref{lem:vieta} applies also to $N$-pole solutions (i.e. with multi-solitons with possibly equal speeds), whereas Proposition~\ref{prop:vari-princ} is restricted to $N$-solitons with different speeds. 
\end{remark}

\begin{proof}[Proof of Lemma~\ref{lem:vieta}]
  Letting $t\to\infty$ in~\eqref{eq:vari-princ}, using the exponential
  localization of each soliton and the asymptotic description of $U$,
for each $j=1,\dots,N$ we have
 \[
H_{N+1}'(Q_{c_j})+\sum_{k=1}^N\lambda_{k}H_k'(Q_{c_j})=0.
\]
Observe here that this argument would not be valid if the functionals
$H_k$ were containing non-local terms. In the present setting, each
$H_k'$ contains only derivatives and potentials based on powers of
$U$ and its spatial derivatives. 

  Recall that each soliton profile $Q_{c_j}$ verifies for each $k\geq
  1$ the equation
 \[
H_{k+1}'(Q_{c_j})=(-c_j)^{k}H_{1}(Q_{c_j}).
\]
As a consequence, we see that for each $j=1,\dots, N$ we have
\[
(-c_j)^{N}+\sum_{k=1}^N\lambda_{k}(-c_j)^{k-1}=0.
\]
In other words,  the speeds $-c_j$ are the roots of the
$N$-th order polynomial with coefficients $1,\lambda_N,\dots,
\lambda_1$.  The relations between the roots of a polynomial and its
coefficients are well-known to be described by Vieta's formulas as in~\eqref{eq:vieta}.
\end{proof}

We will use the following technical result in the course of the proof of Proposition~\ref{prop:vari-princ}.

\begin{lemma}
 \label{lem:olver}
For any $\phi\in H^{N+1}(\R)$
  and for any $j,k=1,\dots,N+1$, we have
\[
\scalar{H_j' (\phi)}{\partial_x H_k'(\phi)}_{L^2}=0
 \]
\end{lemma}

\begin{proof}
  The result is a consequence of the iteration identity~\eqref{eq:recursion}.
  Indeed, for any $\phi\in\mathcal C^\infty_c$ we have
 \begin{multline*}
 \scalar{H_j' (\phi)}{\partial_x H_k'(\phi)}_{L^2}
      = \scalar{H_{j}' (\phi)}{\mathcal K(\phi)H_{k-1}'(\phi)}_{L^2}
      = -\scalar{\mathcal K(\phi)H_{j}' (\phi)}{H_{k-1}'(\phi)}_{L^2}\\
      =-\scalar{\partial_xH_{j+1}' (\phi)}{H_{k-1}'(\phi)}_{L^2}
      = \scalar{H_{j+1}' (\phi)}{\partial_xH_{k-1}'(\phi)}_{L^2}.
 \end{multline*}
    Iterating the process $k-1$ times, we arrive at
 \[
 \scalar{H_j' (\phi)}{\partial_x H_k'(\phi)}_{L^2}=\scalar{H_{j+k-1}' (\phi)}{\partial_xH_{1}'(\phi)}_{L^2}.
 \]
          From the invariance of $H_{j+k-1}$ under translation, we have
 \[
                0=\frac{\rmd H_{j+k-1}\phi(\cdot-y))}{\rmd y}_{|y=0}=\scalar{H_{j+k-1}' (\phi)}{\phi_x}_{L^2}=\scalar{H_{j+k-1}' (\phi)}{\partial_xH_{1}'(\phi)}_{L^2}.
 \]
             Gathering the previous identities leads to the desired conclusion, which by density is also valid in $H^{N+1}(\R)$.
\end{proof}

\begin{proof}[Proof of Proposition~\ref{prop:vari-princ}]
  From Lemma~\ref{lem:vieta}, we know that, if they exist,  $\lambda_1,\dots,\lambda_N$ in Proposition~\ref{prop:vari-princ} are uniquely determined by $c_1,\dots,c_N$ and~\eqref{eq:vieta}. We define
 \[
r(t)=S_N'(U(t)).
\]
By construction, each of the soliton profile $Q_{c_j}$ composing $U$ at the limit $t\to\pm\infty$ is a critical point of $S_N$ and is exponentially decaying, therefore we have
  \begin{multline*}
    S_N'\left(\sum_{j=1}^NQ_{c_j}(\cdot-x_j(t))\right)=\sum_{j=1}^NS_N'(Q_{c_j}(\cdot-x_j(t)))+O\left(e^{-\frac12\sqrt{c_1}|\min_{j,k}(x_j(t)-x_k(t))|}\right)
    \\=O\left(e^{-\frac12\sqrt{c_1}|\min_{j,k}(x_j(t)-x_k(t))|}\right).
  \end{multline*}
Since we have assumed that $c_j\neq c_k$ for $j\neq k$, we can infer from~\eqref{eq:9} that there exists $c_*>0$ such that
\[
  S_N'\left(\sum_{j=1}^NQ_{c_j}(\cdot-x_j(t))\right)=O\left(e^{-c_*|t|}\right).
\]
Hence we can use this result with the expression of $r$ to obtain
 \begin{multline*}
    r(t)=S_N'(U(t))- S_N'\left(\sum_{j=1}^NQ_{c_j}(\cdot-x_j(t))\right)+O\left(e^{-c_*|t|}\right).
 \\=S_N''\left( \sum_{j=1}^NQ_{c_j}(\cdot-x_j(t))\right)\left(U(t)-\sum_{j=1}^NQ_{c_j}(\cdot-x_j(t))\right)+o \left(U(t)-\sum_{j=1}^NQ_{c_j}(\cdot-x_j(t))\right)+O\left(e^{-c_*|t|}\right).
\end{multline*}
By assumption, we have
\[
\norm*{U(t)-\sum_{j=1}^NQ_{c_j}(\cdot-x_j(t))}\lesssim e^{-c_*|t|},
\]
therefore we have
\[
\norm{r(t)}_{L^2}\lesssim e^{-c_* |t|}.
\]
In particular, we have
\[
\lim_{t\to\infty}\norm{r(t)}_{L^2}=0.
\]
Our goal is to show that in fact for all $t\in\R$ we have
\[
r(t)=0.
\]
For this, it is sufficient to show that for some $t_0\in\R$ and for any $v_0\in\mathcal C_c^\infty(\R)$ we have
\[
\scalar{r(t_0)}{v_0}_{L^2}=0.
\]
We choose arbitrarily $t_0\in\mathbb R$ and $v_0\in\mathcal C_c^\infty(\R)$ and consider the evolution problem for the linearized~\eqref{eq:mkdv} equation around $U$ given by
\[
\partial_t v=\partial_x H_2''(U(t))v,\quad v(t_0)=v_0.
\]
We will show that
\[
\partial_t \scalar{r(t)}{v(t)}_{L^2}=0,
\]
and 
\[
\lim_{t\to\infty} \scalar{r(t)}{v(t)}_{L^2}=0,
\]
from which the conclusion follows.

First, we observe that
\[
\partial_t \scalar{r(t)}{v(t)}_{L^2}=\partial_t \scalar{S_N'(U(t))}{v(t)}_{L^2}=\partial_t \scalar{H'_{N+1} (U(t))}{v(t)}_{L^2}+\sum_{j=1}^N\lambda_{j}\partial_t \scalar{H'_j (U(t))}{v(t)}_{L^2}.
\]
We claim that for every $j=1,\dots,N+1$ we have
\[
 \partial_t \scalar{H_j' (U(t))}{v(t)}_{L^2}=0.
\]
Indeed, using the equations verified by $U$ and $v$ (and removing the variable $t$ for convenience) we have
\begin{equation}
 \label{eq:1}
 \partial_t \scalar{H_j' (U)}{v}_{L^2}=\scalar{H_j'' (U)\partial_xH_2'(U)}{v}_{L^2}+\scalar{H_j' (U)}{\partial_x H_2''(U)v}_{L^2}.
\end{equation}
From Lemma~\ref{lem:olver}, we have for any $\phi\in H^{N+1}(\R)$ and for any $j,k=1,\dots,N+1$ that
\[
\scalar{H_j' (\phi)}{\partial_x H_k'(\phi)}_{L^2}=0
 \]
  Writing $\phi=U+sv$ and differentiating in $s$ at $s=0$ we obtain
 \[
\scalar{H_j'' (U)v}{\partial_x H_k'(U)}+\scalar{H_j' (U)}{\partial_x H_k''(U)v}_{L^2}=0
\]
Substituting in~\eqref{eq:1} and using the self-adjointness of  $H_j'' (U) $ we obtain
\[
 \partial_t \scalar{H_j' (U)}{v}_{L^2}=\scalar{H_j'' (U)\partial_xH_2'(U)}{v}_{L^2}-\scalar{H_j'' (U)v}{\partial_x H_2'(U)}_{L^2}=0,
\]
This proves the claim, and we can infer that
\[
\partial_t \scalar{r(t)}{v(t)}_{L^2}=0.
\]
From the exponential decay of $r$, we have
\[
 \scalar{r(t)}{v(t)}_{L^2}\lesssim \norm{v(t)}_{L^2}e^{-c_*|t|}.
\]
Hence if we are able to show that $v$ grows  slower than $e^{c_*t}$, we can readily conclude that necessarily  $\scalar{r(t)}{v(t)}_{L^2}=0$.

To this aim, let us consider a partition of unity constructed in such a way that each member of the partition is (at time infinity) a localizing factor around one of the solitons composing the multi-soliton $U$. The partition that we use is similar to the one used in~\cite{CoLe11,CoMaMe11}. Let $\psi:\R\to\R$ be a $\mathcal C^\infty$ cut-off function defined such that
\[
\psi(s)=0 \text{ if }s\leq-1,\quad 0<\psi(s)<1 \text{ if }-1< s<1,\quad \psi(s)=1 \text{ if }1\leq s.
\]
Define for $j=2,\dots,N$ the middle speeds
\[
m_j=\frac{c_{j-1}+c_j}{2},
\]
Define also  for $(t,x)\in\R\times\R$ the domain walls 
\[
\psi_1(t,x)=1, \quad \psi_j(t,x)=\psi\left(\frac{1}{\sqrt{t}}(x-m_jt)\right),\quad j=2,\dots,N,
\]
and construct the partition of unity as follows:
\[
\phi_j=\psi_j-\psi_{j+1},\quad j=1,\dots,N-1,\quad \phi_N=\psi_N. 
\]
We may now write
\[
v=\sum_{j=1}^N\psi_jv.
\]
Recall (see \cite{HoPeZw11}) the following coercivity property for the linearized action around a $1$-soliton profile $Q_c$: there exists $\delta>0$ such that 
  \begin{equation}
\dual{H_2''(Q_c)w}{w}+c\dual{H_1''(Q_c)w}{w}\geq \delta\norm{w}_{H^1}^2-\frac1\delta\left(\scalar{w}{\partial_x^{-1}\Lambda_c Q_c}^2-\scalar{w}{Q}^2\right).\label{eq:10}
\end{equation}
Observe that $\partial_x^{-1}\Lambda_c Q_c$ and $Q=\partial_x^{-1}\partial_xQ$ form the generalized kernel of the operator
$(H_2''(Q_c)+H_1''(Q_c))\partial_x$ (see the original work of Weinstein \cite{We85} for the equivalent version for Schr\"odinger equations). 
We will use this property on $\psi_jv$ for $j=1,\dots,N$.

We first deal with the orthogonality directions. By direct calculations, we have
\begin{multline*}
  \partial_t\scalar{\psi_jv}{Q_{c_j}(\cdot-x_j(t))}_{L^2}\\= \scalar{\partial_t\psi_j)v}{Q_{c_j}(\cdot-x_j(t))}_{L^2}+ \scalar{\psi_j \partial_t v}{Q_{c_j}(\cdot-x_j(t))}_{L^2}+ \scalar{\psi_jv}{\partial_t x_j(t) \partial_xQ_{c_j}(\cdot-x_j(t))}_{L^2}.
\end{multline*}
  The first term of the right hand side contains a time derivative of $\psi$, hence it will be of order $t^{-\frac12}$. For the second term, we have
      \begin{multline*}
        \scalar{\psi_j \partial_t v}{Q_{c_j}(\cdot-x_j(t))}_{L^2}= \scalar{\psi_j \partial_xH_2''(U)v}{Q_{c_j}(\cdot-x_j(t))}_{L^2}
        \\= -\scalar{\partial_x\psi_j H_2''(U)v}{Q_{c_j}(\cdot-x_j(t))}_{L^2} -\scalar{\psi_j H_2''(U)v}{\partial_x Q_{c_j}(\cdot-x_j(t))}_{L^2}
        \\
        = -\scalar{\psi_j v}{H_2''(U)\partial_x Q_{c_j}(\cdot-x_j(t))}_{L^2} +O\left(t^{-\frac12}\norm{v}_{H^1}\right).
      \end{multline*}
      Moreover, by assumption on $x_j(t)$, the third term gives
        \begin{multline*}
          \scalar{\psi_jv}{\partial_t x_j(t) \partial_xQ_{c_j}(\cdot-x_j(t))}= \scalar{\psi_jv}{c_j\partial_xQ_{c_j}(\cdot-x_j(t))} +O\left(t^{-1}\norm{v}_{L^2}\right)
          \\=\scalar{\psi_jv}{c_jH_1''(U)\partial_xQ_{c_j}(\cdot-x_j(t))} +O\left(t^{-1}\norm{v}_{L^2}\right).
        \end{multline*}
By the localization properties of $\psi_j$, as $t$ is large $U$ is close to the soliton $Q_{c_j}(\cdot-x_j(t))$ on the support of $\psi_j$ and we have
\[
H_2''(U)+c_jH_1''(U)=H_2''(Q_{c_j}(\cdot-x_j(t))+c_jH_1''(Q_{c_j}(\cdot-x_j(t)))+O(e^{-c_*t}). 
\]
Since $\partial_xQ_{c_j}(\cdot-x_j(t))$ is in the kernel of the above operator, this gives
\[
  \partial_t\scalar{\psi_jv}{Q_{c_j}(\cdot-x_j(t))}=O(t^{-\frac12}\norm{v}_{H^1}).
\]
From similar arguments, we may also obtain the result for the other orthogonality direction that we have chosen:
\[
  \partial_t\scalar{\psi_jv}{\partial_x^{-1}\Lambda_{c_j}Q_{c_j}(\cdot-x_j(t))}=O(t^{-\frac12}\norm{v}_{H^1}).
\]
Let $j=1,\dots,N$. We have
    \begin{multline*}
      \partial_t \dual{H_2''(U)\psi_jv}{\psi_jv}=\dual{H_2'''(U)\partial_tU\psi_jv}{\psi_jv}+ 2\dual{H_2''(U)\psi_jv}{\partial_t(\psi_jv)}
      \\=\dual{-6U\partial_tU\psi_jv}{\psi_jv}+ 2\dual{H_2''(U)\psi_jv}{\partial_t\psi_j}+2\dual{H_2''(U)\psi_jv}{\psi_j\partial_tv}.
    \end{multline*}
 We will keep the first term of the right hand side. The second term contains a time derivative of $\psi$, hence it will be of order $t^{-\frac12}$. For the third term,  we have
   \begin{multline*}
     \dual{H_2''(U)\psi_jv}{\psi_j\partial_tv}= \dual{H_2''(U)\psi_jv}{\psi_j\partial_xH_2''(U)v}
     \\= -\dual{\partial_x(H_2''(U)\psi_jv)}{\psi_jH_2''(U)v} -\dual{H_2''(U)\psi_jv}{(\partial_x\psi_j)H_2''(U)v}.
   \end{multline*}
The second term contains a time derivative of $\psi$, hence it will be of order $t^{-\frac12}$. For the first one, we proceed further:
  \begin{multline*}
    \dual{\partial_x(H_2''(U)\psi_jv)}{\psi_jH_2''(U)v}= \dual{\partial_x(H_2''(U)\psi_jv)}{H_2''(U)\psi_jv}-\dual{\partial_x(H_2''(U)\psi_jv)}{(\partial_x\psi_j\partial_xv+\partial_x^2\psi_jv)}\\=-\dual{\partial_x(H_2''(U)\psi_jv)}{(\partial_x\psi_j\partial_xv+\partial_x^2\psi_jv)},
  \end{multline*}
and therefore this term is also of order $t^{-\frac12}$. Summarizing, we have proved that
\[
\partial_t \dual{H_2''(U)\psi_jv}{\psi_jv}=\dual{-6U\partial_tU\psi_jv}{\psi_jv}+O\left(\frac{\norm{\psi_jv}_{H^1}^2}{\sqrt{t}}\right).
\]
We may argue similarly to obtain
\[
 \partial_t \dual{H_1''(U)\psi_jv}{\psi_jv}=\dual{-6U\partial_xU\psi_jv}{\psi_jv}+O\left(\frac{\norm{\psi_jv}_{H^1}^2}{\sqrt{t}}\right).
\]
Hence, we have
\[
  \partial_t\left(\dual{H_2''(U)\psi_jv}{\psi_jv}+c_j\dual{H_1''(U)\psi_jv}{\psi_jv}\right)=\dual{-6U\left(U_t+c_j\partial_xU\right)\psi_jv}{\psi_jv}+O\left(\frac{\norm{\psi_jv}_{H^1}^2}{\sqrt{t}}\right).
\]
Recall that a $1$-soliton $U_{c}$ verifies the following transport equation
\[
\partial_tU_{c}+c\partial_xU_c=0.
 \]
 All we have left to do  is to take into account the localizing factor that we have introduced. Since $\psi_j$ is centered around $c_jt$, by assumption on $U$ we have
 \[
(\partial_tU+c_j\partial_xU)\psi_j=O\left(e^{-c_*t}\right). 
\]
Therefore, using the coercivity property \eqref{eq:10} we have for $t$ large enough
\[
  \partial_t\left(\dual{H_2''(U)\psi_jv}{\psi_jv}+c_j\dual{H_1''(U)\psi_jv}{\psi_jv}\right)\leq \frac{C}{\sqrt{t}}   \left(\dual{H_2''(U)\psi_jv}{\psi_jv}+c_j\dual{H_1''(U)\psi_jv}{\psi_jv}\right),
\]
which gives
\[
\dual{H_2''(U)\psi_jv}{\psi_jv}+c_j\dual{H_1''(U)\psi_jv}{\psi_jv}\lesssim e^{C\sqrt{t}}. 
\]
As a consequence, we have
\[
\norm{v}_{H^1}^2 \lesssim\sum_{j=1}^N \norm{\psi_jv}_{H^1}^2\lesssim e^{C\sqrt{t}},
\]
which implies
\[
\scalar{r(t)}{v(t)}_{L^2}=0. 
  \]
This concludes the proof.
\end{proof}

\section{Inertia Preservation}
\label{sec:spectral-theory}

The tools presented in this section have been developed by Lax~\cite{La75}, Lopes~\cite{Lo03} and Neves and Lopes~\cite{NeLo06}. The
work of Neves and Lopes being devoted to the case of $2$ solitons, we
extended here their results to the case of $N$ solitons with $N$ an
arbitrary integer.

\subsection{The Generalized Sylvester Law of Inertia}
\label{sec:invariance-inertia}

Let $X$ be a real Hilbert space. We first define the inertia of a
self-adjoint operator with positive essential spectrum.

\begin{definition}\label{def:inertia}
  Let $L:D(L)\subset X\to X$ be a self-adjoint operator.
  Assume that there exists $\delta>0$ such that the spectrum of $L$ in $(-\infty,\delta)$ consists into a finite number of eigenvalues with finite geometric multiplicities. 
The \emph{inertia} of $L$, denoted by $\inertia(L)$, is the pair
$(n,z)$, where $n$ is the number of negative eigenvalues of $L$ (counted with geometric multiplicities) and $z$ is
the dimension of the kernel of $L$. 
\end{definition}

We first recall the generalized Sylvester law of inertia, which is the
operator version of the eponymous law for matrices, and can be proved
using the same line of arguments.

\begin{proposition}[Generalized Sylvester Law of Inertia]
 \label{prop:sylvester}
    Let $L:D(L)\subset X\to X$ be a self-adjoint operator such that
the inertia is well-defined,
and let $M$ be a bounded invertible operator. Then we have
\[
\inertia(L)=\inertia(MLM^t),
\]
where $MLM^t$ is the self-adjoint operator with domain $(M^t)^{-1}(D(L))$.
\end{proposition}

\subsection{Iso-inertial Families of Operators}
\label{sec:iso-inertial-family}

We will be working with linearized operators around a
multi-soliton, which fit in the following more generic framework.

Consider the abstract evolution equation
 \begin{equation}
\partial_t u=f(u),\label{eq:abstract-evolution}
\end{equation}
  for $u:\R\to X$, and recall that the following framework was set in
~\cite{La68,Lo03,NeLo06}. Let $X_2\subset X_1\subset X$ be Hilbert spaces
  and $V:X_1\to\R$ be such that the following assumptions are
  verified. 
 \begin{itemize}
 \item [(H1)] The spaces $X_2\subset X_1\subset X$ are continuously
    embedded. The embedding from $X_2$ to $X_1$ is denoted by $i$.
 \item [(H2)] The functional $V:X_1\to\R$ is $\mathcal C^3$.
 \item [(H3)] The function $f:X_2\to X_1$ is $\mathcal C^2$.
 \item [(H4)] For any $u\in X_2$, we have
 \(
V'(i(u))f(u)=0.
 \)
 \end{itemize}
Moreover, given $u\in\mathcal C^1(\R,X_1)\cap \mathcal C(\R,X_2)$ a strong solution of~\eqref{eq:abstract-evolution}, we assume  that there exists a
self-adjoint operator $L(t):D(L)\subset X\to X$ with domain $D(L)$
independent of $t$ such that for $h,k\in Z$, where $Z\subset D(L)\cap
X_2$ is a dense subspace of $X$, we have
\[
\dual{L(t)h}{k}=V''(u(t))(h,k).
 \]
  We consider the operators $B(t) :D(B)\subset X\to X$ such that
  for any $h\in Z$ we have
 \[
B(t)h=-f'(u(t))h,
 \]
and we have the following assumption.
\begin{itemize}
\item [(H5)] The closed operators $B(t)$ and $B^t(t)$ have a common
  domain $D(B)$ which is independent of $t$. The Cauchy problems
 \[
\partial_t u=B(t)u,\quad \partial_t v=B^t(t)v
\]
are well-posed in $X$ for positive and negative times.  
\end{itemize}

We then have the following result, which we recall without proofs (see the work of Lax~\cite{La75} or the work of Lopes~\cite{Lo03}).

\begin{proposition}
 \label{prop:iso-inertia}
  Let $u\in\mathcal C^1(\R,X_1)\cap \mathcal C(\R,X_2)$ be a strong solution of~\eqref{eq:abstract-evolution} and assume that
(H1)-(H5) are satisfied. Then the following assertions hold.
\begin{itemize}
\item \emph{Invariance of the set of critical points.} If there exists $t_0\in\R$ such that $V'(u(t_0))=0$, then
  $V'(u(t))=0$ for any $t\in\R$.
 \item \emph{Invariance of the inertia.} Assume that $u$ is such
    that $V'(u(t))=0$ for all $t\in\R$. Then the inertia     $\inertia(L(t))$ of the
    operator $L(t)$ representing $V''(u(t))$ is independent of $t$.
\end{itemize}
\end{proposition}

\subsection{Iso-inertia at Infinity}
\label{sec:iso-inertia}

Given an $t$-dependent family of operators whose inertia we are interested in,
Proposition~\ref{prop:iso-inertia} allows to choose for a specific $t$
to perform the inertia calculation. This is however in most situations
not sufficient, as we would like to let $t$ go to infinity and relate
the inertia of our family with the inertia of the asymptotic objects
that we obtain. This is what is allowed in the following framework.

Let $X$ be a real Hilbert space. Let $N\in\mathbb N$ and $(\tau^j_n)$ be sequences of
isometries of $X$ for $j=1,\dots,N$. For brevity in notation, we
denote the composition of an isometry $\tau_n^k$ and the inverse of
$\tau_n^j$ by
\[
\tau_n^{k/j}=\tau_n^k(\tau_n^j)^{-1}.
\]
Let $A$, $(B^j)_{j=1,\dots,N}$ be linear operators
and $(R_n)$ be a sequence of linear operators. Define the sequences of
operators based on $(B^j)$ and $(\tau^j_n)$ by
\[
B^j_n=(\tau^j_n)^{-1}B^j\tau^j_n.
\]
We will use the following notations: 
The  resolvent set of an operator $L$ will be denoted by $\rho(L)$.
We denote by $P_{\lambda,\eps}(L)$ the
          spectral projection of $L$ corresponding to the circle of
          center $\lambda\in\mathbb C$ and
          radius $\eps>0$. The range will be denoted by $\operatorname{Range}$
          and the dimension by $\dim$. 

          We make the following assumptions.
\begin{itemize}
\item [(A1)] For all $j=1,\dots,N$ and $n\in\mathbb N$, the operators
  $A$, $A+B^j$, $A+B^j_n$, $A+\sum_{j=1}^N B^j_n+R_n$ are self-adjoint
  with the same domain $D(A)$.
 \item [(A2)] The operator $A$ is invertible.   For all $j=1,\dots,N$
    and $n\in\mathbb N$, the operator $A$ commutes with $\tau^j_n$
    (i.e. $A=(\tau^j_n)^{-1}A\tau^j_n$).
 \item [(A3)]
    There exists $\delta>0$ such that for all $j=1,\dots,N$
      and $n\in\mathbb N$, the spectra of
      $A$, $A+B^j$, $A+B^j_n$, $A+\sum_{j=1}^N B^j_n+R_n$
 in $(-\infty,\delta)$ consist into a finite number of eigenvalues with finite geometric multiplicities. 
 \item  [(A4)]  For every $\lambda \in \cap_{j=1}^N\rho(A+B^j)$ and for all $j=1,\dots,N$ the operators
    $A(A+B^j-\lambda I)^{-1}$  are bounded.
 \item  [(A5)]  In the operator norm, $\norm{R_nA^{-1}}\to0$ as $n\to\infty$.
 \item  [(A6)] For all $u\in D(A)$ and for all $j,l=1,\dots,N$, $j\neq l$ we have
 \[
 \lim_{n\to\infty}\norm{\tau_n^{j/l}B^l \tau_n^{l/j} u}_X\to 0.
 \]
 \item  [(A7)] For all $u\in X$ and for all $j,k=1,\dots, N$, $j\neq k$, we have
        $\tau_n^{j/k} u\rightharpoonup 0$ weakly in $X$ as $n\to \infty$.
 \item  [(A8)]  For all $j=1,\dots,N$, the operator
          $B^jA^{-1}$ is compact.
 \end{itemize}
        Define the operator $L_n:D(A)\subset X\to X$ by
\[
L_n=A+\sum_{j=1}^NB^j_n+R_n.
 \]
We have the following result on the asymptotic behavior of the spectrum of $L_n$ as $n$ goes to infinity. 
 \begin{theorem}\label{thm:asymptotic-inertia}
          Assume that assumptions (A1)-(A8) hold and let
          $\lambda<\delta$. The following assertions hold.
 \begin{itemize}
 \item If $\lambda \in\cap_{j=1}^N\rho(A+B^j)$, then there
            exists $n_\lambda\in\mathbb N$ such that for all $n\geq n_\lambda$ we have
            $\lambda\in\rho(L_n)$.
 \item If $\lambda\in\cup_{j=1}^N\sigma(A+B^j)$, then there
              exists $\eps_0>0$ such that for all $0<\eps<\eps_0$
              there exists $n_\eps\in\mathbb N$ such that for all
              $n\geq n_\eps$
              we
              have
\[
\dim\left(\operatorname{Range}\left(P_{\lambda,\eps} \left(L_n\right) \right)\right)=\sum_{j=1}^N \dim\left(\operatorname{Range}\left( P_{\lambda,\eps} \left(A+B^j\right) \right)\right).
\]
 \end{itemize}
 \end{theorem}

        In our setting, we are interested in particular in the inertia of $L_n$ and we will make use of the following corollary. 
 \begin{corollary}
\label{cor:asymptotic-inertia}
          Under the assumptions of Theorem
~\ref{thm:asymptotic-inertia}, if there exists $n_L$  such
          that for all $n\geq n_L$ we have
 \[
 \dim(\ker(L_n))\geq\sum_{j=1}^N\dim(\ker(A+B^j)),
 \]
          then for all $n\geq n_L$ we have
 \[
\inertia(L_n)=\sum_{j=1}^N\inertia(A+B^j).
 \]
         Moreover, a non-zero eigenvalue of $L_n$ cannot approach $0$
         as $n\to \infty$.  
            
 \end{corollary}

Theorem~\ref{thm:asymptotic-inertia} and Corollary~\ref{cor:asymptotic-inertia} were proved in~\cite{NeLo06} in the
case $N=2$. We adapt here the proof of~\cite{NeLo06} to handle the
case of generic $N\in\mathbb N$. 

\begin{proof}[Proof of Theorem~\ref{thm:asymptotic-inertia}]
  We start by the first assertion. Let $\lambda<\delta$ be such that
  $\lambda \in\cap_{j=1}^N\rho(A+B^j)$. 
By assumption (A3)
  $\lambda$ can either be in the resolvent of $L_n$ or be an eigenvalue with finite multiplicity. Hence, to prove that $\lambda
 \in\rho(L_n)$, it is sufficient to prove that $u=0$ is the only
  solution to
 \[
(L_n-\lambda I)u=0.
\]
Assume therefore that there exists $u\in D(A)$ such that

\begin{equation}
(L_n-\lambda I)u=\left(A+\sum_{j=1}^NB^j_n+R_n-\lambda
  I\right)u=0.\label{eq:3}
\end{equation}
We remark here that since
 \[
(\tau^j_n)^{-1} \left(A+B^j-\lambda I\right)\tau^j_n=A+B^j_n-\lambda I
\]
we have 
 \[
\rho(A+B^j)=\rho(A+B_n^j).
 \]
Since   $\lambda \in\cap_{j=1}^N\rho(A+B^j)$ we may rewrite~\eqref{eq:3} for any $k=1,\dots,N$ as
\[
u=(A+B^k_n-\lambda I)^{-1}\left(-\sum_{\substack{j=1\\j\neq k}}^NB^j_nu-R_nu\right).
\]
We now use this equation recursively and replace the $u$ after
$B^j_n$ by its expression in the right member (with $k$
replaced by $j$) to obtain
\[
u=(A+B^k_n-\lambda I)^{-1}\left(-\sum_{\substack{j=1\\j\neq k}}^NB^j_n (A+B^j_n-\lambda I)^{-1}\left(-\sum_{\substack{l=1\\l\neq j}}^NB^l_nu-R_nu\right)-R_nu\right).
\]
We develop the right member of the previous equation to define the
operator $W_n^k(\lambda):D(A)\to X$ by
 \begin{multline}\label{eq:Wkn}
    W_n^k(\lambda)=(A+B^k_n-\lambda I)^{-1}\left(\sum_{\substack{j=1\\j\neq k}}^NB^j_n (A+B^j_n-\lambda I)^{-1}\left(\sum_{\substack{l=1\\l\neq j}}^NB^l_n+R_n\right)\right)
    -(A+B^k_n-\lambda I)^{-1} R_n.
 \end{multline}
Then $u\in D(A)$ is  a fixed point of $W_n^k(\lambda)$. We aim at
showing that  the operator $W_n^k$ can in fact be extended to a
bounded operator which verifies
$\norm{W_n^k(\lambda)}<1$ for $n$ large. This will imply that $u=0$. We first consider the operator 
\[
(A+B^k_n-\lambda I)^{-1}\sum_{\substack{j=1\\j\neq k}}^NB^j_n (A+B^j_n-\lambda I)^{-1}\sum_{\substack{l=1\\l\neq j}}^NB^l_n.
\]
Since an operator and its adjoint share the same norm and all the
operators that we are manipulating are symmetric by assumption, we have for any $j,k,l=1,\dots,N$, $k\neq j$, $j\neq l$ that
 \begin{equation*}
 \norm{(A+B^k_n-\lambda I)^{-1}B^j_n (A+B^j_n-\lambda I)^{-1}B^l_n}=
 \norm{B^l_n (A+B^j_n-\lambda I)^{-1} B^j_n (A+B^k_n-\lambda I)^{-1}}
\end{equation*}
Since the $\tau^j_n$ are isometries, we have
 \begin{multline*}
 \norm*{  B^l_n (A+B^j_n-\lambda I)^{-1} B^j_n (A+B^k_n-\lambda I)^{-1}}\\
    =
\norm*{(\tau^l_n)^{-1}B^l \tau_n^{l/j}(A+B^j-\lambda I)^{-1}\tau^j_n
  (\tau^j_n)^{-1}B^j\tau_n^{j/k}(A+B^k-\lambda I)^{-1}\tau^k_n}
\\
=
\norm*{  (\tau^j_n)^{-1}\left(\tau_n^{j/l}B^l \tau_n^{l/j}(A+B^j-\lambda I)^{-1}B^j\tau_n^{j/k}(A+B^k-\lambda I)^{-1}\tau_n^{k/j}\right) \tau^j_n}
\\
=
\norm*{ \tau_n^{j/l}B^l \tau_n^{l/j}(A+B^j-\lambda I)^{-1}B^jA^{-1}\tau_n^{j/k}A(A+B^k-\lambda I)^{-1}\tau_n^{k/j}}.
\end{multline*}
Now, by assumption (A4), the family
\[
\tau_n^{j/k}A(A+B^k-\lambda I)^{-1}\tau_n^{k/j}
 \]
  is uniformly bounded.  By assumption (A8), the operator
 \[
B^jA^{-1}
 \]
    is compact. The operator
 \[
      (A+B^j-\lambda I)^{-1}
 \]
    is bounded. And finally, combining all these informations with assumption (A6), we have
 \[
\lim_{n\to\infty}\norm*{ \tau_n^{j/l}B^l \tau_n^{l/j}(A+B^j-\lambda I)^{-1}B^jA^{-1}\tau_n^{j/k}A(A+B^k-\lambda I)^{-1}\tau_n^{k/j}}=0.
 \]
      The terms involving $R_n$ in $W^k_n(\lambda)$ are taken care of
      by assumptions (A4) and (A5): as $n\to\infty$, we have
\[
\norm{(A+B_n^k-\lambda I)^{-1}R_n}=\norm{R_n(A+B_n^k-\lambda I)^{-1}}=\norm{R_nA^{-1}A(A+B_n^k-\lambda I)^{-1}}\to0.
\]
In conclusion, we indeed have
\[
\lim_{n\to\infty}\norm{W^k_n(\lambda)}=0,
\]
which implies that for $n$ large enough $u=0$ is the only solution of~\eqref{eq:3} and that $\lambda \in\rho(L_n)$. This concludes the proof of the first part of Theorem~\ref{thm:asymptotic-inertia}.

We now prove the second part of Theorem~\ref{thm:asymptotic-inertia}. Let $\lambda<\delta$ be such that $\lambda\in \cup_{j=1}^N\sigma(A+B^j)$. By isolatedness of the eigenvalues below $\delta$, there exists $\eps_0>0$ such that for all $\mu\in \mathbb C$ verifying $|\lambda-\mu|\leq \eps_0$, $\mu\neq \lambda$ we have $\mu\in \cap_{j=1}^N\rho(A+B^j)$. Take now $0<\eps<\eps_0$. By the first part, there exists $n_\eps$ such that for all $\mu\in \mathbb C$ verifying $|\lambda-\mu|= \eps$, $\mu\neq \lambda$, we have $\mu\in \rho(L_n)$. We denote by $\Gamma\subset \mathbb C$ the circle centered at $\lambda$ with radius $\eps$. The corresponding spectral projection is then given by
\begin{equation*}
P_{\lambda,\eps}(L_n)=\frac{1}{2\pi i}\int_{\Gamma}(L_n-\mu
I)^{-1}d\mu.
\end{equation*}
    We use a strategy similar to the one of the first part to express the resolvent $(L_n-\mu I)^{-1}$. Assume that $u\in D(A)$ and $f\in X$ are such that
 \[
(L_n-\mu I)^{-1}f=u.
\]
It is equivalent to 
\[
(L_n-\mu I)u=Au+\sum_{j=1}^NB^j_nu+R_nu-\mu u=f.
\]
Since for all $k=1,\dots,N$ we have $\mu\in\rho(A+B^k)$, we may write
\[
u=(A+B^k_n-\mu I)^{-1}\left(f-\sum_{\substack{j=1\\j\neq k}}^NB^j_nu-R_nu\right).
\]
As in the first part, we use the equation recursively to replace the $u$ after $B^j_n$ to get
\[
u=(A+B^k_n-\mu I)^{-1}\left(f-\sum_{\substack{j=1\\j\neq k}}^NB^j_n (A+B^j_n-\mu I)^{-1}\left(f-\sum_{\substack{l=1\\l\neq j}}^NB^l_nu-R_nu\right)-R_nu\right).
\]
Using the operator $W^k_n$ already defined in the first part (see~\eqref{eq:Wkn}), we write
\[
u=W^k_n(\mu)u+(A+B^k_n-\mu I)^{-1}f+(A+B^k_n-\mu I)^{-1}\sum_{\substack{j=1\\j\neq k}}^NB^j_n (A+B^j_n-\mu I)^{-1}f.
 \]
  We already proved in the first part that $\lim_{n\to \infty}\norm{W^k_n(\mu)}=0$, therefore if $n$ is large enough we may write $u$ as the image of
  $f$ by the following operator, therefore giving a new expression for
  the resolvent:
 \begin{equation*}
    (I-W^k_n(\mu))^{-1}\left((A+B^k_n-\mu I)^{-1}+(A+B^k_n-\mu
      I)^{-1}\sum_{\substack{j=1\\j\neq k}}^NB^j_n (A+B^j_n-\mu
      I)^{-1}\right)
    =(L_n-\mu I)^{-1}.
 \end{equation*}
Let us define an approximate projection by 
\[
P_n=\frac{1}{2\pi i}\int_{\Gamma}(A+B^k_n-\mu I)^{-1}d\mu
+\frac{1}{2\pi i}\int_{\Gamma} (A+B^k_n-\mu
I)^{-1}\sum_{\substack{j=1\\j\neq k}}^NB^j_n (A+B^j_n-\mu I)^{-1}d\mu.
\]
Since  $\lim_{n\to \infty}\norm{W^k_n(\mu)}=0$, we have
\[
\lim_{n\to\infty}\norm{P_{\lambda,\eps}(L_n)-P_n}=0.
\]
As $P_{\lambda,\eps}(L_n)$ has finite dimensional range, this implies
that for $n$ large enough we have
\begin{equation*}
\dim\left(\operatorname{Range}\left(P_{\lambda,\eps} \left(L_n\right) \right)\right)=\dim\left(\operatorname{Range}\left( P_n \right)\right).
\end{equation*}
We now analyze $P_n$. The first term in the expression of $P_n$ is
just
\[
\frac{1}{2\pi i}\int_{\Gamma}(A+B^k_n-\mu I)^{-1}d\mu=P_{\lambda,\eps}(A+B^k_n).
\]
Moreover, we have
 \begin{equation*}
 \dim\left(\operatorname{Range}
      (P_{\lambda,\eps}(A+B^k_n))\right)
    =\dim\left(\operatorname{Range}
      ((\tau_n^k)^{-1}
      P_{\lambda,\eps}(A+B^k)\tau_n^k)
 \right)
    =\dim\left(\operatorname{Range}(
      P_{\lambda,\eps}(A+B^k))\right).
 \end{equation*}
Remark here that it may very well be that $\lambda\not\in\sigma(A+B^k)$ and $P_{\lambda,\eps}(A+B^k)$ has null range. 

For the second term in the expression of $P_n$, we argue as follows. For $j\neq k$, we have
 \begin{equation*}
    (A+B^k_n-\mu I)^{-1}B^j_n (A+B^j_n-\mu
    I)^{-1}
=(\tau_n^j)^{-1}\tau_n^{j/k}(A+B^k-\mu
    I)^{-1}\tau_n^{k/j} B^j (A+B^j-\mu I)^{-1}\tau_n^j.
 \end{equation*}
We will therefore analyze the operator
\begin{equation}
\label{eq:4}
Q_{n,k,j}(\mu)=\tau_n^{j/k}(A+B^k-\mu
    I)^{-1}\tau_n^{k/j} B^j (A+B^j-\mu I)^{-1}.
 \end{equation}
It is well-known (see e.g.~\cite[III. \S 6. 4. and V. \S 3. 5.]{Ka76}) that the resolvent of a self-adjoint operator $U$ around an isolated eigenvalue $\lambda$ verifies 
\[
(U-\mu I)^{-1}=\frac{P_{\lambda}}{\lambda-\mu}+(U-\lambda I)^{-1}(I-P_\lambda)+U(\mu),
\]
where $P_{\lambda}$ is the orthogonal projection on the eigenspace corresponding to $\lambda$ and $U(\mu)$ is holomorphic in $\mu$ and verifies $U(\lambda)=0$.
Applying this to $A+B^l$ for $l=j,k$, we get
\[
(A+B^l-\mu I)^{-1}=\frac{P^l}{\lambda-\mu}+(A+B^l-\lambda I)^{-1}(I-P^l)+U^l(\mu)
\]
where we have used the notation $P^l=P_{\lambda,\eps}(A+B^l)$ and $U^l(\mu)$ is   holomorphic in $\mu$ and verifies $U^l(\lambda)=0$. Consequently, we have 
 \begin{multline*}
  Q_{n,k,j}(\mu)
  =
\tau_n^{j/k}
\left( 
\frac{P^k}{\lambda-\mu}+(A+B^k-\lambda)^{-1}(I-P^k)+U^k(\mu)
 \right)
\tau_n^{k/j} \\
\times  B^j \times 
\\
\left(
\frac{P^j}{\lambda-\mu}+(A+B^j-\lambda)^{-1}(I-P^j)+U^j(\mu)
\right).
\end{multline*}
The residue of  the operator $Q_{n,k,j}$ given by~\eqref{eq:4} at $\lambda$ is thus given by
 \begin{equation}
\label{eq:2}
\tau_n^{j/k}
P^k
\tau_n^{k/j} 
B_j
(A+B^j-\lambda)^{-1}(I-P^j)
+
\tau_n^{j/k}
 (A+B^k-\lambda)^{-1}(I-P^k)
\tau_n^{k/j} 
B^jP^j.
\end{equation}

The second term in~\eqref{eq:2}  is treated in the following
way. Since $P^j$ projects on the kernel of $A+B^j-\lambda I$, we have 
 \begin{multline*}
    B^jP^j
=-(A-\lambda I)P^j 
=-(A+\tau_n^{j/k} B^k \tau_n^{k/j} -\lambda I)P^j+\tau_n^{j/k} B^k \tau_n^{k/j} P^j 
\\= -\tau_n^{j/k} (A+B^k   -\lambda I) \tau_n^{k/j} P^j+\tau_n^{j/k} B^k \tau_n^{k/j} P^j.
 \end{multline*}
Therefore, we have
\begin{multline*}
 \tau_n^{j/k}
 (A+B^k-\lambda)^{-1}(I-P^k)
\tau_n^{k/j} 
B^jP^j
=-P^j
+\tau_n^{j/k}
P^k
\tau_n^{k/j} 
P^j+\tau_n^{j/k}
 (A+B^k-\lambda)^{-1}(I-P^k)
B^k \tau_n^{k/j}  P^j.
\end{multline*}
We claim that, as $n$ tends to infinity, only the term $-P^j$ will remain. Indeed, let $(\xi_k^p)_{p=1,\dots,P}$ and $(\xi_j^q)_{q=1,\dots,Q}$ be normalized bases for the (finite dimensional) subspaces on which $P^k$ and $P^j$ project. Given $u\in X$, we have
\[
\tau_n^{j/k}
P^k
\tau_n^{k/j} 
P^j u=
\sum_{\substack{p=1,\dots,P\\q=1,\dots,Q}}\scalar{\xi_j^q}{u}_{X}\scalar{\tau_n^{j/k}\xi_k^p }{\xi_j^q}_X \tau_n^{k/j}\xi_j^q.
\]
Therefore, we have
\[
\norm{\tau_n^{j/k}
P^k
\tau_n^{k/j} 
P^j }\lesssim \sum_{\substack{p=1,\dots,P\\q=1,\dots,Q}}\scalar{\tau_n^{j/k}\xi_k^p }{\xi_j^q}_X
\]
By assumption (A7), the right hand side goes to $0$ as $n\to\infty$. In addition, since $P^j$ has finite range and $ (A+B^k-\lambda)^{-1}(I-P^k)
$ is bounded, by assumption (A6), as $n\to\infty$, we have 
\[
\norm{\tau_n^{j/k} (A+B^k-\lambda)^{-1}(I-P^k)
B^k \tau_n^{k/j}  P^j}\to 0,
\]
which proves our claim. 

The first term in~\eqref{eq:2} will vanish as $n\to\infty$ as we now show.  By assumption (A4), the operator 
\[
A
(A+B^j-\lambda)^{-1}(I-P^j)
\]
is bounded (note that here assumption (A4) remains valid even if $\lambda\in\sigma(A+B^j)$ as we are projecting out the spectral subspace associated with $\lambda$). By assumption (A8), the operator
\[
B_jA^{-1}A
(A+B^j-\lambda)^{-1}(I-P^j)
\]
is compact, which combined with assumption (A7) shows that as $n\to\infty$ we have
\[
\norm{\tau_n^{j/k}
P^k
\tau_n^{k/j} 
B_j
(A+B^j-\lambda)^{-1}(I-P^j)}\to 0. 
\]

Summarizing the previous analysis, we have shown that 
\[
\lim_{n\to\infty}\norm[\bigg]{P_n-(\tau_n^k)^{-1} P^k \tau_n^k-\sum_{\substack{j=1,\dots,N\\j\neq k}}(\tau_n^j)^{-1} P^j \tau_n^j}=0. 
\]
Therefore, for $n$ large enough we have 
\[
\dim\left(\operatorname{Range}\left( P_n \right)\right)=\sum_{j=1,\dots,N}\dim\left(\operatorname{Range}\left( P^j \right)\right).
\]
This concludes the proof.
\end{proof}

\section{Spectral Analysis}
\label{sepecsec}

In the theory of stability of solitary waves (as developed e.g. in~\cite{GrShSt87,We85}  or more recently in~\cite{deGeRo15}), it is customary
to use the coercivity properties of a linearized operator around the
solitary wave to obtain the stability estimate. If the perturbation is
set at the level of the solitary wave profile, the
corresponding linearized operator is independent of time. When trying
to adopt a similar strategy for multi-solitons, it is not possible to
write the perturbation at the level of a profile 
independent of time and the linearized operator is necessarily time
dependent.

The combination of two main
arguments allows to overcome this difficulty. First, we have shown in Section~\ref{sec:spectral-theory}
that a form of iso-spectrality holds for linearized operators
around a multi-soliton, in the sense that the inertia (i.e. the number
of negative eigenvalues and the dimension of the kernel, see
Definition~\ref{def:inertia} below) is preserved along the
time evolution. Second, at large time, the linearized operator can be
viewed as a composition of several decoupled linearized operators
around each of the soliton profiles composing the multi-soliton, and
the spectrum of the multi-soliton linearized operator will converge to
the union of the spectra of the linearized operators around each
soliton.

\subsection{The auxiliary operators $M_c$ and $M_c^t$}

Let $c>0$ and consider the associated soliton profile $Q_c$  given in~\eqref{eq:Q}. 
We introduce an auxiliary linear operator $M_c$ and its adjoint $M_c^t$, defined as follows:
\begin{equation}
 \label{Mt}
 \begin{gathered}
    M_c, M_c^t: D(M_c)=D(M_c^t)=H^1(\R)\subset L^2(\R)\rightarrow L^2(\R),\\
    M_ch(x)=h'(x)+\sqrt{c}\tanh(\sqrt{c}x) h(x), \quad M_c^tk(x)=-k'(x)+\sqrt{c}\tanh(\sqrt{c}x) k(x).
 \end{gathered}
\end{equation}
The operators $M_c$ and $M_c^t$ are linked with $Q_c$ by the following observation. Given $h,k\in H^1(\R)$, we have
 \begin{equation}
M_ch=Q_c\partial_x\left(\frac h{Q_c}\right),\quad M_c^tk=-\frac1Q_c\partial_x(Q_ck).\label{eq:6}
\end{equation}
The operators $M_c$ and $M_c^t$ are linked to Darboux transformations and the factorization of Schr\"odinger operators. As such, their use is not limited to integrable equations and they appear in other contexts, see in particular~\cite[Section 3.2]{ChGuNaTs07}.
The auxiliary operators $M_c$ and $M_c^t$ verify the following properties (see e.g.~\cite[Lemma 5]{NeLo06}).

\begin{lemma}\label{lem:3.2.} Let $M_c,M_c^t$ be given by~\eqref{Mt}. The following properties are verified.
 \begin{itemize}
 \item The operators $M_c$ and $M_c^t$ map odd functions on even functions and even functions on odd functions.
 \item The null space of $M_c$ is spanned by $Q_c$ and $M_c^t$
    is injective.
 \item The operator $M_c$ is surjective and the image of $M_c^t$ is the $L^2(\R)$-subspace orthogonal to $Q_c$.
 \end{itemize}
\end{lemma}

\begin{proof}
  That $M_c$ and $M_c^t$ map odd (resp. even) functions to even (resp. odd) functions is easily seen from their definition in~\eqref{Mt}, using in particular the oddness of $x\mapsto\tanh(x)$.

  Let $h\in H^1(\R)$ be such that $M_ch=0$. From the expression of $M_c$ in terms of $Q_c$ given in~\eqref{eq:6}, this implies that $h/Q_c$ is constant, i.e. $h$ is a multiple of $Q_c$. Hence we indeed have $\ker(M_c)=\Span(Q_c)$.

  Let $k\in H^1(\R)$ be such that $M_c^tk=0$. From the expression of $M_c^t$ in terms of $Q_c$ given in~\eqref{eq:6}, this implies that $Q_ck$ is constant, i.e. $k$ is a multiple of $1/Q_c$. However, $1/Q_c$ does not belong to $H^1(\R)$, hence $k=0$. This gives the injectivity of $M_c^t$.

  From the preceding observations combined with the fact that $M_c^t$ is the adjoint of $M_c$, we have
 \[
\overline{\operatorname{im}(M_c)}=\ker(M_c^t)^\perp=L^2(\R),\quad \overline{\operatorname{im}(M_c^t)}=\ker(M_c)^\perp=Q_c^\perp. 
\]
It remains to prove that both images are closed.

We start with $\operatorname{im}(M_c)$. 
Let $g\in L^2(\R)$. We look for $h\in H^1(\R)$ such that $M_ch=g$. To this aim, we define the operator $\mathcal T$ by
\[
\mathcal Tg(x)=Q_c(x)\int_0^x\frac {g(y)}{Q_c(y)}\rmd y.
\]
We clearly have
\[
(\mathcal Tg)'-\frac{Q_c'}{Q_c}\mathcal Tg=(\mathcal Tg)'+\sqrt{c}\tanh(\sqrt{c}x) \mathcal Tg=g,
\]
hence we only have to prove that $(\mathcal Tg)\in L^2(\R)$ to prove that $(\mathcal Tg)\in H^1(\R)$ and $M_c(\mathcal Tg)=g$. 
We will prove the operator $\mathcal{T}$ is bounded in $L^1(\R)$ and $L^\infty(\R)$ respectively, thus in $L^2(\R)$ by interpolation. Recall the explicit expression of $Q_c$ given in~\eqref{eq:Q}: $Q_c(x)=\sqrt{2c}\sech(\sqrt{c}x)$. Hence, we have
\[
 \left|Q_c(x)\int_0^x \frac{\rmd y}{Q_c(y)}\right|=\left|\frac{1}{\sqrt{c}}\sinh (\sqrt{c}x)\sech(\sqrt{c}x)\right|=\left|\frac{1}{\sqrt{c}}\tanh(\sqrt{c}x)\right|
\]
and we see that $\mathcal{T}$ is bounded in $L^\infty(\R)$. We now prove that  $\mathcal{T}$ is bounded in $L^1(\R)$.  Let $a>0$ and $g\in L^1(\R)$. By integration by parts, we have
\begin{multline*}
 \int_0^aQ_c(x)\int_0^x \frac{|g(y)|}{Q_c(y)}\rmd y\rmd x
  =
 \int_0^a\partial_x\left(-\int_x^aQ_c(s)ds\right)\int_0^x \frac{|g(y)|}{Q_c(y)}\rmd y\rmd x
 \\
 = \frac{1}{\sqrt{c}}\int_0^a\left(\arctan\big(\sinh (\sqrt{c}a)\big)-\arctan\big(\sinh (\sqrt{c}x)\big)\right)\cosh(\sqrt{c}x)|g(x)|\rmd x\\
 = \frac{1}{\sqrt{c}}\int_0^a\left(\arctan\big(\sinh (\sqrt{c}a)\big)-\frac\pi2+\arctan\left(\frac{1}{\sinh (\sqrt{c}x)}\right)\right)\cosh(\sqrt{c}x)|g(x)|\rmd x\\
\leq\frac{1}{\sqrt{c}}\int_0^a\arctan\left(\frac{1}{\sinh (\sqrt{c}x)}\right)\cosh(\sqrt{c}x)|g(x)|\rmd x\leq C\int_0^a|g(x)|\rmd x,
\end{multline*}
where we have used the famous calculus formula
\[
\arctan(x)+\arctan\left(\frac1x\right)=\frac\pi2.
 \]
The case $a<0$ can be treated in a similar way. This shows the boundedness of $\mathcal{T}$ in $L^1(\R)$. By interpolation, $\mathcal T$ is also bounded in $L^2(\R)$.

We now consider $\operatorname{im}(M_c^t)$. 
Let $g\in L^2(\R)$ be such that $\scalar{g}{Q_c}_{L^2}=0$. We look for $k\in H^1(\R)$ such that $M_c^tk=g$. Using~\eqref{eq:6}, we define
\[
\mathcal Sg=k(x)=-\frac{1}{Q_c(x)}\int_{-\infty}^xg(y)Q_c(y)\rmd y.
\]
From similar arguments as before,  the operator $\mathcal S$ is bounded in $L^2$ and verifies $M_c^t\mathcal Sg=g$, which concludes the proof. 
\end{proof}

The operators $M_c$ and $M_c^t$ have remarkable algebraic properties. We give the simplest ones in the following lemma. 

\begin{lemma}
 \label{lem:prop_M}
  The following identities hold
  \begin{align}
    M_cM_c^t&=-\partial_x^2+c,&
    M_c^tM_c&=-\partial_x^2+c-Q_c^2,\\
    M_c(-\partial_x^2-2Q_c\partial_x^{-1}(Q_c\partial_x))&=(-\partial_x^2-Q_c^2)M_c,&
    (-\partial_x^2-Q_c^2)M_c^t&=M_c^t(-\partial_x^2),\label{Mt7}\\
    M_cQ_c&=0,&
    M_c^tQ_c&=-2(Q_c)_x,\label{Mt8}\\
    M_c(xQ_c)&=Q_c,&
    M_c^t(xQ_c)&=-Q_c-2x(Q_c)_x.\label{eq:MtxQ}
 \end{align}
\end{lemma}
Each of the identities of Lemma~\ref{lem:prop_M} may be obtain by elementary calculations. We omit the details here. 

\subsection{Spectra of linearized operators around $1$-soliton profiles}
\label{sec:spectra-line-oper}

Let $N\in\mathbb N$ and $0<c_1\leq \cdots \leq c_N$. Denote by $1,\lambda_N,\dots,\lambda_1$ the coefficients of the polynomial whose roots are $(-c_j)$ (see~\eqref{eq:vieta}). Let $S_N$ be the corresponding functional defined in~\eqref{eq:def-S_N}.  For any $j=1,\dots,N$, define operators $L_{N,j}:H^N(\R)\subset L^2(\R)\to L^2(\R)$ by
 \begin{equation*}
  L_{N,j}:=S_N''(Q_{c_j}).
\end{equation*}
For brevity, we use the notation
\[
M_j:=M_{c_j},\quad M_j^t:=M_{c_j}^t.
 \]
The main interest of the auxiliary operators $M_j$ and $M_j^t$  stems from the following result, which gives a factorization of $L_{N,j}$ in terms of pure differential operators.

\begin{proposition}
 \label{prop:factorization_L_N_j}
For any $j=1,\dots,N$, the operator $L_{N,j}$ verifies the following factorization
 \[
M_jL_{N,j}M_j^t=M_j^t\left(\prod_{k=1}^{N}(-\partial_x^2+c_k)\right)M_j.
 \]
 \end{proposition}

  The proof of Proposition~\ref{prop:factorization_L_N_j} relies on several ingredients. We first prove the result for $N=1$. Then we establish an iteration identity at the level of the conserved quantities linearized around soliton profiles and use it to factorize the operators $L_{N,j}$. Finally, we obtain the conclusion by combining these elements with the properties of $M_j$ and $M_j^t$. 

  We start with the case $N=1$. By direct calculations, we have the
  following result (which has been used in particular in~\cite{Wa17b}). 
\begin{lemma}
 \label{lem:identities}
  The operator $L_{1,1}$ is given by
 \[
L_{1,1}=H_2''(Q_{c_1})+c_1H_1''(Q_{c_1})=-\partial_x^2+c_1-3Q_{c_1}^2.
 \]
  The following operator identity holds:
 \begin{align}
    M_1L_{1,1}M_1^t&=M_1^t\left(-\partial_x^2+c_1\right)M_1.\label{Mt1}
 \end{align}
\end{lemma}

 \begin{remark}
  It would also be possible to obtain by direct calculations the result for $N=2$.
However, even for $N=3$ the calculations are becoming very intricate and it would not be reasonable to calculate by hand any further. 
\end{remark}

 \begin{lemma}
 \label{lem:6.1.}
    Let $Q_c$ be a soliton profile of~\eqref{eq:mkdv} with speed $c>0$ as given in~\eqref{eq:Q}. For any $n\in\mathbb N$, and for any $z\in H^n(\R)$ we have
\begin{equation}\label{eq:recursion Hamiltonian}
 H''_{n+1}(Q_c)z=\mathcal R(Q_c)H''_{n}(Q_c)z+(-1)^nc^{n-1}\left(Q_c^2z+2Q_c\partial_x^{-1}(Q_c'z)\right),
\end{equation}
where the \emph{recursion operator} $\mathcal R(Q_c)$ is defined by
 \begin{equation*}
\mathcal R(Q_c)=-\partial_x^2-2Q_c\partial_x^{-1}(Q_c\partial_x).
\end{equation*}
\end{lemma}

\begin{proof}
  The strategy of the proof is to linearize the recursion identity~\eqref{eq:recursion} around $Q_c$.
  Let $n\in\mathbb N$, $n\geq 1$, and $z\in H^n(\R)$. We have by  differentiation of~\eqref{eq:recursion} around $Q_c$ at $z$ the following identity:
 \begin{equation*}
   \partial_x\left(H''_{n+1}(Q_c)z\right)=\mathcal{K}(Q_c)(H''_{n}(Q_c)z)+\left(\mathcal{K}'(Q_c)z\right)H_n'(Q_c),
 \end{equation*}
    where
 \[
 \mathcal{K}'(Q_c)z=-4Q_cz\partial_x-2z_x\partial_x^{-1}(Q_c\partial_x)-2(Q_c)_x\partial_x^{-1}(z\partial_x).
 \]
  Observe that the operator $\mathcal K(Q_c)$ might be rewritten in the following way
 \begin{equation*}
\mathcal{K}(Q_c)=-\partial_x^3-2Q_c^2\partial_x-2(Q_c)_x\partial_x^{-1}(Q_c\partial_x)=\partial_x\left( -\partial_x^2-2Q_c\partial_x^{-1}(Q_c\partial_x) \right)=\partial_x\mathcal R(Q_c)
\end{equation*}
  From the variational principle~\eqref{eq:1-sol variaprinciple} satisfied by the $1$-soliton profile $Q_c$, we have
 \[
    H_n'(Q_c)=(-c)^{n-1}H_1'(Q_c)=(-c)^{n-1}Q_c,
 \]
  hence
 \[
 \left(\mathcal{K}'(Q_c)z\right)H_n'(Q_c)=(-c)^{n-1} \left(\mathcal{K}'(Q_c)z\right)Q_c.
 \]
Moreover, we have
 \begin{multline*}
 \left(\mathcal{K}'(Q_c)z\right)Q_c=-4Q_c(Q_c)_xz-2z_x\partial_x^{-1}(Q_c(Q_c)_x)-2(Q_c)_x\partial_x^{-1}(z(Q_c)_x)\\
    =-\left(2Q_c(Q_c)_xz+Q_c^2z_x\right)-2\left(Q_c(Q_c)_xz+(Q_c)_x\partial_x^{-1}(z(Q_c)_x)\right) 
                                                                                                        =-\partial_x\left(\left(Q_c^2z\right)+2\left(Q_c\partial_x^{-1}(z(Q_c)_x)\right)\right). 
\end{multline*}
Combining the previous identities and removing the  $\partial_x$ give the desired recursion identity and conclude the proof.
\end{proof}

\begin{lemma}
 \label{lem:factor}
  Fix $j=1,\dots,N$.
  The operator $L_{N,j}$ can be factorized in the following way:
 \begin{equation}
L_{N,j}=\left(\prod_{k=1,k\neq j}^{N}(\mathcal R(Q_{c_j})+c_k)\right)(H_2''(Q_{c_j})+c_j H_1''(Q_{c_j})).\label{eq:7}
\end{equation}
\end{lemma}

\begin{proof}
The proof proceeds by finite induction. Let $k=1,\dots,N$, $k\neq j$. We have

 \[
 L_{N,j}=H_{N+1}''(Q_{c_j})+\sum_{l=1}^{N-1} \tilde \lambda_l H''_{l+1}(Q_{c_j})+c_k\tilde L_{N-1,j},
\]
where $\tilde \lambda_l$ is obtained from $\lambda_l$ by removing all terms containing $c_k$ and
\[
\tilde L_{N-1,j}:=\tilde S_{N-1}''(Q_{c_j}):=H_{N-1}''(Q_{c_j})+\sum_{l=1}^{N-1}\tilde\lambda_l H''_l(Q_{c_j}).
\]
Writing more explicitly the coefficients $\tilde \lambda_l$:
\[
\tilde\lambda_1=c_1+\cdots+c_{k-1}+c_{k+1}+\cdots+c_N,\quad \dots,\quad \tilde\lambda_{N-1}=c_1\cdots c_{k-1}c_{k+1}\cdots c_N,
\]
we observe that $(1,\tilde\lambda_1,\dots,\tilde\lambda_{N-1})$ is the family of coefficients of the polynomial with roots $-c_1,\dots, -c_{k-1},-c_{k+1},\dots,-c_N$.  
We now use the recursion formula~\eqref{eq:recursion Hamiltonian} to obtain
 \begin{multline*}
    H_{N+1}''(Q_{c_j})+\sum_{l=1}^{N-1} \tilde \lambda_l H''_{l+1}(Q_{c_j})
 \\
    =\mathcal R(Q_{c_j})\left(H_{N}''(Q_{c_j})+\sum_{l=1}^{N-1} \tilde \lambda_l H''_{l}(Q_{c_j})\right)-\left((-c_j)^{N-1}+\sum_{l=1}^{N-1}(-c_j)^{l-1}\tilde\lambda_l\right)(Q_{c_j}^2+2Q_{c_j}\partial_x^{-1}((Q_{c_j})_x\cdot)
 \\
    =\mathcal R(Q_{c_j})\tilde L_{N-1,j},
 \end{multline*}
  where we have used the fact that $-c_j$ is a root of  the polynomial of coefficients $1,\tilde\lambda_1,\dots,\tilde\lambda_{N-1}$ (recall that $j\neq k$). Gathering the previous calculations, we obtain the following formula:
 \[
 L_{N,j}=(\mathcal R(Q_{c_j}) +c_k)\tilde L_{N-1,j}.
\]
Iterating the process for any $k=1,\dots,N$, $k\neq j$, we obtain the desired formula~\eqref{eq:7}.
\end{proof}

With Lemmas~\ref{lem:identities},~\ref{lem:6.1.} and~\ref{lem:factor} in hand, we may now proceed to the proof of Proposition~\ref{prop:factorization_L_N_j}.

\begin{proof}[Proof of Proposition~\ref{prop:factorization_L_N_j}]
Using successively~\eqref{eq:7},~\eqref{Mt7} (first equation),~\eqref{Mt1} and~\eqref{Mt7} (second equation) we have
 \begin{multline*}
    M L_{N,j}M^t 
      =\left(\prod_{k=1,k\neq j}^{N}(-\partial_x^2-Q_{c_j}^2+c_k)\right)M(H_2''(Q_{c_j})+c_j H_1''(Q_{c_j}))M^t \\
        =\left(\prod_{k=1,k\neq j}^{N}(-\partial_x^2-Q_{c_j}^2+c_k)\right)M^t(-\partial_x^2+c_j)M 
        =M^t\left(\prod_{k=1}^{N}(-\partial_x^2+c_k)\right)M.
 \end{multline*}
      This concludes the proof.
 \end{proof}

\begin{lemma}
 \label{lem:7.3.}
  For $j=1,\dots,N$, the operator $L_{N,j}$ verifies the following properties.
 \begin{itemize}
 \item The essential spectrum of $L_{N,j}$ is $[c_1\cdots c_N,\infty)$. 
 \item If there does not exist $k$ such that $c_k=c_j$, then we have the following. 
 \begin{itemize}
 \item  The operator $L_{N,j}$ has zero as a simple eigenvalue with eigenvector $(Q_{c_j})_x$. 
 \item If $j$ is odd, then $L_{N,j}$ has exactly one negative eigenvalue.
 \item If $j$ is even, then $L_{N,j}$ has no negative eigenvalue.
 \end{itemize}
 \item If there exists $k$ such that $c_k=c_j$, then the operator $L_{N,j}$ has zero as a double eigenvalue with eigenvectors $(Q_{c_j})_x$ and $\Lambda Q_{c_j}$ (see~\eqref{eq:def-Lambda}) and the rest of the spectrum is positive.
 \end{itemize}
\end{lemma}

\begin{remark}
 \label{rmk:7.4.}
  As a particular case of Lemma~\ref{lem:7.3.}, we obtain  the spectrum of the linearized operator $L_{N,j}$ around the $1$-soliton with profile $Q_{c_j}$. This information might be used to obtain the nonlinear stability of $1$-solitons of~\eqref{eq:mkdv} (see e.g.~\cite{BoLiNg04}).
\end{remark}

\begin{proof}[Proof of Lemma~\ref{lem:7.3.}]
  Since $Q_{c_j}$ is smooth and exponentially decaying, the operator $L_{N,j}$ is a compact perturbation of
 \[
 \prod_{k=1}^{N}(-\partial_x^2+c_k).
 \]
From Weyl's Theorem, they share the same essential spectrum , which is $[c_1\cdots c_N,\infty)$.

Given $c>0$, introduce the scaling derivative $\Lambda Q_c$, given by
 \begin{equation}
 \label{eq:def-Lambda}
\Lambda Q_c:=\frac{\rmd Q_{\tilde c}}{\rmd \tilde c}_{|\tilde c=c}=\frac{1}{2c}(Q_c+x(Q_c)_x).
\end{equation}

By construction, each soliton profile $Q_{c_j}$  verifies the variational principle~\eqref{eq:vari-princ}, i.e. $S_N'(Q_{c_j})=0$. Differentiating with respect to $x$ and $c_j$ readily gives
\[
  L_{N,j}(\partial_xQ_{c_j})=0,
\]
Using $H_k'(Q_{c_j})=(-c_j)^{k-1}H_1'(Q_{c_j})=(-c_j)^{k-1}Q_{c_j}$, we have
\[
  L_{N,j}\Lambda Q_{c_j}=-\sum_{k=1}^N\frac{\partial\lambda_k}{\partial c_j}H_k'(Q_{c_j})=-\sum_{k=1}^N\frac{\partial\lambda_k}{\partial c_j}(-c_j)^{k-1}Q_{c_j}=-\left(\prod_{k=1,k\neq j}^N(c_k-c_j)\right)Q_{c_j}.
\]
Observe that if there is any $k$ such that $c_k=c_j$, then $\Lambda Q_{c_j}\in\ker(L_{N,j})$.

These preliminary observations being made, we now proceed to the proof. 

Any $z\in H^{N}(\R)$ might be decomposed orthogonally as
\[
z=a Q_{c_j}+M_j^tg
\]
for $a\in\R$ and $g\in H^{N+1}(\R)$.

The operator $L_{N,j}$ preserves the symmetry (i.e. if $z$ is even, then $L_{N,j}z$ is also even), hence it is natural
to distinguish between two cases : $z$ odd or $z$ even. 

We first treat the case where $z$ is odd. In this case, $a=0$ and (see Lemma~\ref{lem:3.2.}) $g$ is even.
 We have
\begin{multline*}
 \left\langle L_{N,j}z,z\right\rangle
  =\left\langle L_{N,j}M_j^tg,M_j^tg\right\rangle
  =\left\langle M_jL_{N,j}M_j^tg,g\right\rangle\\
  =\left\langle M_j^t \prod_{k=1}^{N}(-\partial_x^2+c_k)
 M_jg,g\right\rangle
    =\left\langle \prod_{k=1}^{N}(-\partial_x^2+c_k)
 M_jg,M_jg\right\rangle.
\end{multline*}
In particular, $ \left\langle L_{N,j}z,z\right\rangle>0$, unless $M_jg=0$, i.e. $g$ is a multiple of $Q_{c_j}$. Since $M_j^tQ_{c_j}=-2(Q_{c_j})_x$  (see~\eqref{Mt8}), and $L_{N,j}(\partial_xQ_{c_j})=0$, this implies that $0$ a simple (for odd functions) eigenvalue of $L_{N,j}$, with associated eigenvector $(\partial_x Q_{c_j})$.

We then treat the case where $z$ is even. In this case, we may have $a\neq0$ and  (see Lemma~\ref{lem:3.2.}) $g$ is odd. Recall from~\eqref{eq:MtxQ} that
\(
M_j^t(xQ_{c_j})=-Q_{c_j}-2x(Q_{c_j})_x.
\)
Therefore, we may rewrite $z$ as
\[
z=4ac_j\Lambda Q_{c_j}+M_j^tk,\quad k=axQ_{c_j}+g.
\]
This gives
\begin{multline}
 \label{eq:8}
 \left\langle L_{N,j}z,z\right\rangle
  =16a^2c_j^2\left\langle L_{N,j}\Lambda Q_{c_j},\Lambda Q_{c_j}\right\rangle+8ac_j\left\langle L_{N,j}\Lambda Q_{c_j},M_j^tk\right\rangle+\left\langle L_{N,j}M_j^tk,M_j^tk\right\rangle
 \\
  =-16a^2c_j^2 \left(\prod_{k=1,k\neq j}^N(c_k-c_j)\right)\left\langle Q_{c_j},\Lambda Q_{c_j}\right\rangle
  -8ac_j \left(\prod_{k=1,k\neq j}^N(c_k-c_j)\right)\left\langle Q_{c_j},M_j^tk\right\rangle
  +\left\langle L_{N,j}M_j^tk,M_j^tk\right\rangle
 \\
  =-16a^2c_j^{\frac32} \left(\prod_{k=1,k\neq j}^N(c_k-c_j)\right)+\left\langle \prod_{k=1}^{N}(-\partial_x^2+c_k)
 M_jk,M_jk\right\rangle,
\end{multline}
where we have used
\begin{equation*}
\left\langle Q_{c_j},\Lambda Q_{c_j}\right\rangle=\frac12\frac{\rmd }{\rmd c}_{|c=c_j}\norm{Q_c}_{L^2}^2=\frac{\rmd }{\rmd c}_{|c=c_j} H_1(Q_c)=\frac{\rmd }{\rmd c}_{|c=c_j} (2\sqrt{c})=\frac{1}{\sqrt{c_j}}.
\end{equation*}

To proceed further, we distinguish between two cases. First, we assume that if $k\neq j$, then $c_k\neq c_j$.

When $j$ is even, since we have $0<c_1<\cdots<c_N$,~\eqref{eq:8} implies that $ \left\langle L_{N,j}z,z\right\rangle>0$ unless $a=0$ and $M_j^tk=0$, i.e. $z=0$.

When $j$ is odd,~\eqref{eq:8}  implies that $ \left\langle L_{N,j}z,z\right\rangle>0$ on the hyperplane $\{a=0\}$, hence $L_{N,j}$ can have at most one nonnegative eigenvalue. Using $\Lambda Q_{c_j}$ as a test function, we have
\[
 \left\langle L_{N,j}\Lambda Q_{c_j}, \Lambda Q_{c_j}\right\rangle=
  -c_j^{-\frac12} \left(\prod_{k=1,k\neq j}^N(c_k-c_j)\right)<0,
\]
which implies the existence of a negative eigenvalue.

Finally, assume that there exists $k\neq j$ such that $c_k=c_j$. In this case,~\eqref{eq:8} implies that $ \left\langle L_{N,j}z,z\right\rangle>0$ unless $M_j^tk=0$, i.e. $z=4ac_j\Lambda Q_{c_j}$, which makes $\Lambda Q_{c_j}$ the unique possible direction for the $0$ eigenvalue. This concludes the proof.
\end{proof}

\section{Stability of Multi-Solitons}\label{stasec7}

This section is devoted to the proof of Theorem~\ref{thm:stability}. To this aim, we will show that multi-solitons of~\eqref{eq:mkdv} verify a stability criterion established by Maddocks and Sachs~\cite{MaSa93}. Before stating the stability criterion, we introduce some notation. Recall that a $N$-soliton solution $U^{(N)}(t,x)\equiv U^{(N)}(t,x;{\mathbf{c}},{\mathbf{x}})$ defined in~\eqref{eq:def-multi-solitons} is a critical point of an associated action functional $S_N$ defined in~\eqref{eq:def-S_N}.

In general, the $N$-soliton $U^{(N)}$ is not a minimum of $S_N$. At best, it is a constrained (and non-isolated) minimizer of the following variational problem
\begin{equation*}
\min H_{N+1}(u) \quad\quad \text{subject to} \quad  H_{j}(u)= H_{j}(U^{(N)}), \quad j=1,2,...,N.
\end{equation*}

We define the self-adjoint operator
\[
 \mathcal L_N:=S_N''(U^{(N)})
\]
and denote by
\[
  n(\mathcal L_N)
\]
  the number of negative eigenvalues of $\mathcal L_N$. Observe that the above defined objects are a priori time-dependent. We also define a $N\times N$ Hessian matrix by
\[
D:=\left\{\frac{\partial^2S_N(U^{(N)})}{\partial \lambda_i\partial \lambda_j}\right\},
\]
and denote by
\[
  p(D)
\]
  the number of positive eigenvalues of $D$. Since $S_N$ is a
  conserved quantity for the flow of~\eqref{eq:mkdv}, the matrix $D$
  is independent of $t$.  The proof of Theorem~\ref{thm:stability} relies
  on the following theoretical result, which was obtained by Maddocks
  and Sachs~\cite[Lemma 2.3]{MaSa93}. 

\begin{proposition}\label{prop:7.1.}
Suppose that
\begin{equation}\label{n=p}
n(\mathcal L_N)=p(D).
\end{equation}
Then there exists  $C>0$ such that $U^{(N)}$ is a non-degenerate unconstrained
minimum of the augmented Lagrangian (Lyapunov function)
\begin{equation*}
S_N(u)+\frac{C}{2}\sum_{j=1}^N\left(H_j(u)-H_{j}(U^{(N)})\right)^2.
\end{equation*}
As a consequence, $U^{(N)}$ is nonlinearly stable. 
\end{proposition}

Hence, to complete the proof of Theorem~\ref{thm:stability}, it is sufficient to verify~\eqref{n=p}.

We start with the count of the number of positive eigenvalues of the Hessian matrix $D$.

\begin{lemma}
 \label{lem:7.2.}
  For all
finite values of the parameters ${\mathbf{c}},{\mathbf{x}}$ with $0<c_1 < \cdots<c_N$, we have
\[
  p(D)= \left\lfloor\frac{N+1}{2}\right\rfloor.
 \]
\end{lemma}

\begin{proof}
  Let $t$ be fixed. For notational convenience, we omit the dependency in $t$ in the proof (as the result will be in any case independent of $t$). For any $1\leq i,j\leq N$, we have
\[
  D_{ij}
  =\frac{\partial^2S_N}{\partial \lambda_i \partial \lambda_j}
  =\sum_{k=1}^N \frac{\partial c_k}{\partial \lambda_i} \frac{\partial}{\partial c_k}\frac{\partial S_N}{\partial \lambda_j}
  =\sum_{k=1}^N \frac{\partial c_k}{\partial \lambda_i} \frac{\partial H_j}{\partial c_k},
\]
where we have used the fact that
\[
 \frac{\partial S_N}{\partial \lambda_j}=\dual{S_N'(U^{(N)})}{\frac{\partial U^{(N)}}{\partial \lambda_j}}+H_j(U^{(N)})=H_j(U^{(N)}).
\]
We observe that $D$ can be obtained as a product of two matrices:
\[
D=AB,\quad A=\left(\frac{\partial c_j}{\partial \lambda_i} \right),\quad B=\left(\frac{\partial H_j}{\partial c_i},\right)
 \]
The value of $H_j$ is explicitly known (see~\eqref{eq:jHamliton}) for each $Q_{c_j}$ composing the asymptotic form of the multi-soliton $U^{(N)}$. Therefore, we have 
 \[
\frac{\partial H_j(U^{(N)})}{\partial c_i}=(-1)^j\frac{2}{2j+1}\frac{\partial }{\partial c_i}\sum_{k=1}^N c_k^{\frac{2j+1}{2}}=(-1)^j c_i^{\frac{2j-1}{2}}.
\]
The value of $c_j$ in terms of  the coefficients $\lambda_k$ cannot be easily expressed. However, we may express $\lambda_k$ in terms of $c_j$ using Vieta's formula~\eqref{eq:vieta}. We therefore have an explicit expression for the inverse of $A$:
\[
A^{-1}=
\begin{pmatrix}
  1&c_2+c_3+\cdots+c_N&\cdots &c_2c_3\cdots c_N \\
  1&c_1+c_3+\cdots+c_N&\cdots&c_1c_3\cdots c_N\\
 \vdots&\vdots&& \vdots\\
1&  c_1+c_2+\cdots+c_{N-1}&\cdots&c_1c_2\cdots c_{N-1}
\end{pmatrix}.
 \]
  Observe that
 \[
A^{-1}D(A^{-1})^t=B(A^{-1})^t,
 \]
    and therefore, by Sylvester's law of inertia (see Proposition~\ref{prop:sylvester}), the number of positive eigenvalues of $D$ is the same as the number of positive eigenvalues for $B(A^{-1})^t$, which turns out to be very simple.  Indeed, the entries of the $j$-th column of $(A^{-1})^t$ are the coefficients of a polynomial whose roots are $-c_1,\dots,-c_{j-1},-c_{j+1},\dots,-c_N$ and the entries of the $i$-th line of $B$ can be rewritten as $(\sqrt{c_j})^{-1}(-c_j)^j$. Hence $B(A^{-1})^t$
is a diagonal matrix with diagonal entries given by
 \[
(-1)^{N-1}\frac{1}{\sqrt{c_j}}\prod_{k\neq j}(c_j-c_k).
 \]
      The number of positive entries is
 \[
\left\lfloor \frac{N+1}{2}\right\rfloor,
 \]
which is the desired result. 
 \end{proof}

Now we verify that $n(\mathcal L_N)$ is also equal to $\left\lfloor\frac{N+1}{2}\right\rfloor$. In fact we can go further and we prove the following. 
\begin{lemma}\label{lem:inertia-L}
  The operator $\mathcal L_N$ verifies
  \begin{equation*}
 \inertia(\mathcal L_N)=\left(n(\mathcal L_N),z(\mathcal L_N)\right)=\left(\left\lfloor\frac{N+1}{2}\right\rfloor,N\right).
 \end{equation*}
\end{lemma}

From the preservation of inertia stated in Theorem~\ref{thm:asymptotic-inertia}, we know that
\[
\inertia(\mathcal L_N)=\sum_{j=1}^N \inertia(L_{N,j}).
 \]
Therefore, Lemma~\ref{lem:inertia-L} is a direct consequence of the results of Section~\ref{sec:spectra-line-oper}, in particular Lemma~\ref{lem:7.3.}.

\bibliographystyle{abbrv}
\bibliography{../master}

\end{document}